\documentclass[11 pt]{amsart}
\usepackage[margin=2.5cm]{geometry}
\usepackage{amsfonts,graphicx,enumerate}
\usepackage{amssymb}
\usepackage{amsmath}
\usepackage{latexsym}
\usepackage{mathrsfs}
\usepackage{amsthm}
\usepackage{verbatim}
\usepackage{amscd}
\usepackage{hyperref}

\usepackage{comment}
\usepackage{framed}
\usepackage{color}
\definecolor{shadecolor}{gray}{0.875}

\newtheorem{thrm}{Theorem}[section]
\newtheorem{lem}[thrm]{Lemma}
\newtheorem{cor}[thrm]{Corollary}
\newtheorem{prop}[thrm]{Proposition}
\newtheorem{conj}[thrm]{Conjecture}

\theoremstyle{definition}
\newtheorem{defn}[thrm]{Definition}

\newtheorem{exmple}[thrm]{Example}
\newtheorem{rmk}[thrm]{Remark}
\newtheorem{ques}[thrm]{Question}%[section]
\newtheorem{notn}[thrm]{Notation}

\newtheorem*{caut}{Caution}

\DeclareMathOperator{\Eff}{\overline{Eff}}

\DeclareMathOperator{\Upsef}{Upsef}

\DeclareMathOperator{\Nef}{Nef}

\DeclareMathOperator{\Mov}{\overline{Mov}}

\DeclareMathOperator{\Supp}{Supp}

\DeclareMathOperator{\pl}{PL}

\DeclareMathOperator{\bpf}{BPF}

\DeclareMathOperator{\upsef}{Upsef}

\input xy
\xyoption{all}

\newcommand{\factor}[2]{\left. \raise 1pt\hbox{\ensuremath{#1}} \right/
        \hskip -2pt\raise -3pt\hbox{\ensuremath{#2}}}

\begin{document}

\title{Positive cones of dual cycle classes}

\author{Mihai Fulger}
\address{Department of Mathematics, Princeton University\\
Princeton, NJ \, \, 08544}
\address{Institute of Mathematics of the Romanian Academy, P. O. Box 1-764, RO-014700,
Bucharest, Romania}
\email{afulger@princeton.edu}

\author{Brian Lehmann}
\thanks{The second author is supported by NSF Award 1004363.}
\address{Department of Mathematics, Boston College  \\
Chestnut Hill, MA \, \, 02467}
\address{Department of Mathematics, Rice University \\
Houston, TX \, \, 77005}

\email{lehmannb@bc.edu}

\begin{abstract}
We study generalizations for higher codimension cycles of several well-known definitions of the nef cone of divisors on a projective variety.  These generalizations fix some of the pathologies exhibited by the classical nef cone of higher codimension classes. As an application, we recover the expected properties of the cones $\Eff_k(X)$ for all $k$. %The properties of another generalization, the analogue of basepoint freeness, were essential in other work by the authors in \cite{fl13z} and \cite{fl14k}.
\end{abstract}

\maketitle

\section{Introduction}

A fundamental invariant of projective algebraic geometry is the cone of nef divisors $\Nef(X)$. By \cite{kleiman66}, it admits several equivalent characterizations: it is the dual of the Mori cone of curves $\overline{\rm NE}(X)$, the closure of the cone of ample line bundle classes, and the closure of the cone generated by classes of divisors in basepoint free linear series.%, or of elements of basepoint free linear series. 

In higher codimension the picture is more subtle.  Let $X$ be a projective variety over an algebraically closed field and let $N_{k}(X)$ denote the numerical group of dimension $k$-cycles with $\mathbb{R}$-coefficients.  The \emph{pseudoeffective} cone $\Eff_k(X)$ is the closure in $N_k(X)$ of the cone generated by classes of $k$-dimensional subvarieties of $X$. Then $\Nef^k(X)$ is defined for smooth $X$ as the dual of $\Eff_k(X)$ with respect to the intersection pairing. When $X$ is singular, we work instead in the space of \emph{dual cycle classes}, the abstract dual $N^k(X)$ of $N_k(X)$.

Interestingly, nef classes do not generally share the other positivity properties exhibited by nef divisors. Indeed \cite{delv11} constructs examples on abelian varieties of nef classes of codimension two that are not even pseudoeffective.  Guided by the alternative characterizations for the nefness of divisors, in this paper we construct ``positive cones'' inside the spaces $N^k(X)$. These are geometric generalizations of nefness in higher codimension which are better suited for applications.  These cones are contained in $\Nef^k(X)$ and satisfy the following properties (which we will see are also satisfied by the nef cone):
\begin{enumerate}
\item they are \emph{full-dimensional} (i.e. span $N^k(X)$) and \emph{salient} (i.e. do not contain lines);
\item they contain the complete intersections of ample divisors in their strict interior;
\item they are preserved by pullbacks.
\end{enumerate}
Furthermore, these cones have the following advantages over the nef cone:
\begin{enumerate} \setcounter{enumi}{3}
\item they are contained in $\Eff^{k}(X)$;
\item they are preserved by the intersection product $N^{k}(X) \times N^{r}(X) \overset{\cap}{\longrightarrow} N^{k+r}(X)$.
\end{enumerate}
%We say that a collection of convex cones in $N^k(X)$ is \emph{good} if it satisfies each of these properties.

A \emph{positive cone of dual classes} is a subcone of $N^k(X)$ satisfying properties (1)-(5) above.

\subsection{The pliant cone} A globally generated divisor class is the pullback of an effective divisor from a projective space.  The analogous notion in higher codimension is to pullback effective cycle classes from Grassmann varieties instead.  We define the \emph{pliant cone} $\pl^k(X)$ to be the closure of the cone generated by products of such classes with total codimension $k$.

\begin{exmple}(cf. \ref{ex:grassmann}) If $X$ is a Grassmann variety of dimension $n$, then $\pl^{n-k}(X)=\Eff_k(X)$.
\end{exmple}

\begin{exmple} \cite{delv11} analyzes two types of abelian varieties in detail: a product $E^{\times n}$ where $E$ is a complex elliptic curve with CM, and $A \times A$ for a very general complex abelian surface $A$.  In both cases the pliant cone coincides with the effective cone (in every codimension).  It would be interesting to describe the pliant cone for other abelian varieties.
\end{exmple}

\begin{thrm}(cf. \ref{lem:pl generates}, \ref{rmk:plpull}, \ref{lem:pliantupsef}, \ref{lem:ci intpliant}) Let $X$ be a projective variety over an algebraically closed field, and let $k\geq 0$.  Then $\pl^{k}(X)$ is a positive cone of dual classes.
%Then $\pl^k(X)$ is a good subcone of $\Nef^k(X)\subset N^k(X)$.
\end{thrm}

The main difficulty is proving that complete intersections belong to the strict interior of $\pl^k(X)$. % which appears to be a new result even for $\Nef^k(X)$, at least in this generality. We mention that the consideration of Grassmann varieties instead of only projective spaces in the definition of pliancy guarantees the full-dimensionality of $\pl^k(X)$.
As an application, we verify that $\Eff_k(X)$ has the expected properties suggested by the case of the Mori cone, i.e. $k=1$. 

\begin{thrm}(cf. \ref{lem:ci big}, \ref{cor:effsalient}, \ref{cor:norms}, \ref{cor:surjeff}) Let $X$ be a projective variety over an algebraically closed field, and let $k\geq 0$. Then
\begin{enumerate}
\item $\Eff_k(X)$ is a full-dimensional and salient subcone of $N_k(X)$.
\item Complete intersections of dimension $k$ of ample classes are contained in the strict interior of $\Eff_k(X)$. %and in the strict interior of $\Nef^{k}(X)$.
\item For any ample divisor class $h$, the function $\deg_h:\Eff_k(X)\to\mathbb R_{\geq 0}$ defined by $\alpha\mapsto\alpha\cdot h^k$ is the restriction to $\Eff_k(X)$ of a norm on $N_k(X)$.
\item If $\pi:X\to Y$ is a surjective morphism of projective varieties then $\pi_*\Eff_k(X) = \Eff_k(Y)$. %If $\pi$ is onto, then equality holds.
\end{enumerate}
\end{thrm}

\noindent 
The subtlety of the theorem is the treatment of the pseudoeffective classes that are not effective, but only limits of effective classes.  While some cases of this theorem are certainly known (see for example \cite[Lemma 2.2]{djv13}), surprisingly a proof in this generality seems to have been missing from the literature. %that $\Eff_k(X)$ is salient seems to have been missing from the literature.  Over $\mathbb C$, it is implied by \cite[Lemma 2.2]{djv13}. Similar statements are proved in \cite[Proposition 1.3]{bfj09} and \cite[Lemma 2.3]{chms13} for Cartier divisors.

In Definition \ref{def:pliant} we give an equivalent definition for pliancy in terms of characteristic classes of globally generated vector bundles on $X$. From this perspective, the pliant cone appears implicitly in the work of Fulton and Lazarsfeld \cite{fl83} on the positivity of characteristic classes of nef vector bundles.

\subsection{The basepoint free cone} A basepoint free linear series of divisors gives a family of divisors on $X$ such that for a fixed subvariety $Y\subset X$, the general member of the family intersects $Y$ properly. Inspired by this we say that an effective class $\alpha\in\Eff_{n-k}(X)$ is \emph{strongly basepoint free} if there exists a projective morphism $p:U\to W$ with equidimensional fibers of dimension $(n-k)$ onto a \emph{quasiprojective} variety $W$ and a \emph{flat} map $s:U\to X$ such that $(s|_F)_*[F]=\alpha$ where  $F$ is a general fiber of $p$. For $X$ smooth, we define the \emph{basepoint free cone} $\bpf^k(X)$ in $N^k(X)$ to be the closure of the cone generated by strongly basepoint free classes.  We emphasize that basepoint freeness is naturally a ``contravariant'' property preserved by pull-back so that $N^{k}(X)$ is the right ambient space for the cone.

\begin{exmple}(cf.~\ref{ex:bpfhom}) If $X$ is a smooth projective homogeneous space under the transitive action of a connected algebraic group $G$, then $\bpf^k(X)=\Eff^k(X)$ for all $k$.\qed
\end{exmple}

\begin{exmple}(cf.~\ref{ex:bpfandmds}) Let $X$ be a smooth Mori Dream Space of dimension $n$.  Then $\bpf^{n-1}(X) = \Nef^{n-1}(X)$. (The existence of many small modifications is an important part of the proof; we do not know how to characterize the basepoint free cone of curves for arbitrary smooth $X$.)
\end{exmple}

Basepoint freeness turns out to be surprisingly versatile.  Its properties were instrumental in other work by the authors in \cite{fl13z} and \cite{fl14k}. 

\begin{thrm}(cf.~\ref{lem:bpfflatpush}, \ref{lem:maintechnicalbpf}, \ref{lem:plbpfupsef}, \ref{cor:bpffullsal}) Let $X$ be a \emph{smooth} projective variety over an algebraically closed field, and let $k\geq 0$. Then $\bpf^{k}(X)$ is a positive cone of dual classes, and in addition,
\begin{enumerate} \setcounter{enumi}{5} %[i)]
%\item $\bpf^k(X)$ is a good cone in $N^k(X)$.
\item $\pl^k(X)\subseteq\bpf^k(X)\subseteq\Nef^k(X)$.
\item If $\pi:Y\to X$ is a flat morphism of relative dimension $d$ from a smooth projective variety $Y$, and $\alpha\in\bpf^{k+d}(Y)$, then $\pi_*\alpha\in\bpf^k(X)$.
\end{enumerate}
\end{thrm}

\noindent We do not know if the flat pullback of cycles descends to numerical equivalence.  If it does, then one can naturally define the cone $\bpf^{k}(X) \subset N^{k}(X)$ for any projective variety $X$ and all the properties above will still hold.

\subsection{The universally pseudoeffective cone} 
If $\xi$ is a nef divisor class on $X$, and $\pi:Y\to X$ is a morphism of projective varieties, then $\pi^*\xi$ is a pseudoeffective divisor class on $Y$. In fact this property determines the nefness of $\xi$. Inspired by this we say that $\alpha\in N^k(X)$ is \emph{universally pseudoeffective} if $\pi^*\alpha$ is pseudoeffective for any morphism of projective varieties $\pi:Y\to X$ with $Y$ smooth. (The definition also makes sense when $Y$ is singular, but requires more care.) These classes form a closed convex cone denoted $\upsef^k(X)$. 

By letting $\pi$ range through inclusions of $k$-dimensional subvarieties in $X$, we see that a universally pseudoeffective class is nef. Pulling back by the identity of $X$, it follows that universally pseudoeffective classes are pseudoeffective, hence in view of the examples in \cite{delv11}, the inclusion $\upsef^k(X)\subset\Nef^k(X)$ may be strict.

\begin{thrm}(cf. \ref{rmk:upsefgood}, \ref{rem:upsefnef}, \ref{prop:domupsef}, \ref{lem:plbpfupsef}) Let $X$ be a projective variety of dimension $n$ over an algebraically closed field, and let $k\geq 0$. Then $\upsef^{k}(X)$ is a positive cone of dual classes, and in addition,
\begin{enumerate} \setcounter{enumi}{5} %[i)]
%\item $\upsef^k(X)$ is a good cone in $N^k(X)$.
\item $\pl^k(X)\subseteq\upsef^k(X)\subseteq\Nef^k(X)$.
\item When $X$ is smooth, then $\pl^k(X)\subseteq\bpf^k(X)\subseteq\upsef^k(X)\subseteq(\Nef^k(X)\cap\Eff_{n-k}(X))$.
\item Suppose $\pi: Y \to X$ is a flat morphism from a projective variety $Y$ of relative dimension $d$, and that $X$ is smooth.  If $\alpha \in \Upsef^{k+d}(Y)$ then $\pi_{*}\alpha \in \Upsef^{k}(X)$.
\item Suppose $\pi: Y \to X$ is a dominant morphism from a projective variety $Y$ and $\alpha \in N^{k}(X)$.  If $\pi^{*}\alpha$ is universally pseudoeffective, then $\alpha$ is as well.
\end{enumerate}
\end{thrm}

While a priori weaker than pliancy or basepoint freeness, universal pseudoeffectivity is easier to compute.

\begin{exmple}
\begin{itemize}
\item[(\ref{ex:upsefcurves})] If $X$ is a smooth projective variety of dimension $n$, then $\upsef^{n-1}(X)=\Nef^{n-1}(X)=\Mov_1(X)$, where the latter is the movable cone of curves in the sense of \cite{bdpp04}.
\item[(\ref{ex:sphericalupsef})] If $X$ is a smooth spherical (e.g. toric) variety of dimension $n$, then $\Upsef^k(X)=\Nef^k(X)\subseteq\Eff_{n-k}(X)$ for all $k$.
\item[(\ref{ex:projbundleupsef})] The same conclusion holds if $X$ is a projective bundle over a smooth projective curve.
\item[(\ref{ex:flexupsef})] If $S$ is a smooth projective surface, and $F$ is a rank-two ample vector bundle on $S$, then the zero section of the total space of $F$ sitting as an open subset in $X=\mathbb P(\mathcal O\oplus F^{\vee})$ is in the strict interior of $\upsef^2(X)$.  
\end{itemize}
\end{exmple}

We also give a simpler criterion for testing universal pseudoeffectivity.

\begin{prop}(cf. \ref{prop:bircritforbpf}) Let $X$ be a projective variety over an algebraically closed field. A class $\alpha\in N^k(X)$ is universally pseudoeffective if and only if $\pi^*\alpha$ is pseudoeffective for any projective morphism $\pi:Y\to X$ that is generically finite \emph{onto its image}.
\end{prop}

The maps $\pi$ need not be dominant. In characteristic zero we may replace ``generically finite'' with ``birational'' in the above (cf. Remark \ref{rmk:birforupsef}). 

\subsection{Comparisons}

The following table encapsulates the properties of dual positive classes.  In our opinion, all the properties except the last are essential.

\vskip.5cm
\begin{center}
\begin{tabular}{ l | c | c | c | c}
   & pliancy & \begin{tabular}{c} basepoint \\ freeness \end{tabular} & \begin{tabular}{c} universal \\ pseudo-effectiveness \end{tabular} & nefness \\ \hline
\begin{tabular}{l} defines a \\ full-dimensional \\ salient cone? \end{tabular} & \checkmark & \checkmark & \checkmark & \checkmark \\ \hline
\begin{tabular}{l} preserved by \\ pullback? \end{tabular} & \checkmark  & \checkmark & \checkmark & \checkmark  \\ \hline
\begin{tabular}{l} contains complete \\ intersections in \\ interior of cone? \end{tabular} & \checkmark  & \checkmark & \checkmark & \checkmark  \\ \hline
\begin{tabular}{l} preserved by \\ intersection product? \end{tabular}& \checkmark  & \checkmark & \checkmark & \text{\sffamily X} \\ \hline
\begin{tabular}{l} contained in \\ $\Eff^{k}(X)$? \end{tabular} & \checkmark  & \checkmark & \checkmark & \text{\sffamily X}  \\ \hline
\begin{tabular}{l} preserved by \\ flat pushforward \\ to smooth base? \end{tabular} & ?  & \checkmark & \checkmark & \checkmark \\ \hline
\begin{tabular}{l} can be checked \\ after pullback by \\ surjective morphism? \end{tabular} & ?  & ? & \checkmark & \checkmark
\end{tabular}
\end{center}
\vskip.5cm

For $X$ smooth we have containments $\pl^k(X)\subseteq\bpf^k(X)\subseteq\upsef^k(X)\subseteq \Nef^k(X)$.  We have seen that the last inclusion may be strict, and the following important example shows that we may have strict inequalities $\pl^{k}(X) \subsetneq \upsef^{k}(X)$ and $\bpf^{k}(X) \subsetneq \upsef^{k}(X)$.

\begin{exmple}
\cite{bh15} constructs a toric fourfold $X$ and a nef surface class $\alpha$ on $X$ which give a counterexample to a question of Demailly concerning positive currents.  Their construction also shows that $\alpha$ is not basepoint free (and hence not pliant).  Since every nef class on a toric variety is universally pseudo-effective, we have $\bpf^{2}(X) \subsetneq \upsef^{2}(X)$.  (See Example \ref{upsefnotbpf} for details.)
\end{exmple}

The definition and study of pliancy and basepoint freeness are motivated by their applications, while universal pseudoeffectivity is an important intersection theoretic positivity property that they share.
It is interesting to ask how these positivity notions interact with  
other versions in the literature (for example, Hartshorne's definition  
via the ampleness of the normal bundle of an l.c.i subscheme (\cite{har70}) and Ottem's  
extension (\cite{ott12})). We discuss these connections more in Section \ref{s:questions}, together with some variations of pliancy and universal pseudoeffectivity.

\subsection*{Organization} In Section \ref{backgroundsection} we set up notation, recalling in particular the definition of the numerical groups $N_k(X)$.  We establish basic properties of the dual spaces $N^k(X)$, showing in particular that they are generated by polynomials in Chern classes of vector bundles on $X$. We also give an overview of the known properties of $\Eff_k(X)$ and $\Nef^k(X)$. Section \ref{s:pliant} is dedicated to the study of the pliant cone. As an application we recover the expected properties of $\Eff_k(X)$ that seem to have been missing from the literature. The properties and examples of the universally pseudoeffective cone are illustrated in \S\ref{s:upsef}, while the properties of the basepoint free cone, essential to the work of the authors in \cite{fl13z} and \cite{fl14k}, are described in \S\ref{s:bpf}. We end with a list of open questions in Section \ref{s:questions}.

\subsection*{Acknowledgments} We thank June Huh, Conner Jager, Alex K\" uronya, Robert Lazarsfeld, and John Christian Ottem for useful conversations. 

\section{Background and preliminaries} \label{backgroundsection}
Throughout we will work over an algebraically closed ground field $K$ of arbitrary characteristic.  A variety is an irreducible reduced scheme of finite type over $K$.   %A cycle will always mean an $\mathbb{R}$-cycle unless otherwise qualified.

\subsection{Cycles and dual cycles}
\label{ssec:cycles}

A \textit{cycle} on a projective variety $X$ is a finite formal linear combination 
$Z=\sum_i a_iV_i$ of closed subvarieties of $X$. We use the denominations \textit{integral}, \textit{rational}, or \textit{real} when the coefficients are $\mathbb Z$, $\mathbb Q$, or $\mathbb R$ respectively. When all $V_i$ have dimension $k$, we say that $Z$ is a \textit{$k$-cycle}. When for all $i$ we have $a_i\geq 0$, we say that the cycle is \textit{effective}.
To any closed \emph{subscheme} $V\subset X$ we associate its \textit{fundamental integral cycle} $[V]$ as in \cite[\S1.5]{fulton84}.

The group of integral $k$-cycles is denoted $Z_k(X)$. Its rank is usually infinite. 
In order to study the geometry of cycles on $X$, several equivalence relations have been
introduced on $Z_k(X)$. One example is rational equivalence; the rational equivalence classes form
the Chow group $A_k(X)$, which may still have infinite rank.

\subsubsection{Characteristic classes of vector bundles} \label{charclasssec}
For any vector bundle $E$ of rank $e$ on $X$ and any non-negative integer $i$, \cite[\S 3.1]{fulton84} constructs a Segre class $s_{i}(E^{\vee})$ which is a graded linear operator on the Chow groups $A_k(X)$.  \cite[\S2.5]{fulton84} first defines a first Chern class action for line bundles.
Then 
$$s_i(E^{\vee})\cap\alpha:=\pi_*(c_1^{i+e-1}(\mathcal O_{\mathbb{P}(E)}(1))\cap \pi^*\alpha)\in A_{k-i}(X)$$
for any $\alpha\in A_k(X)$, 
where $\pi:\mathbb P(E)\to X$ is the projection map of the projective bundle of quotients associated to $E$, 
%where $\mathcal O_E(1)$ is the relative Serre line bundle, 
and %where
$\pi^*$ is the flat pullback of \cite[\S1.7]{fulton84}.  The dual of $E$ appears since \cite{fulton84} works with projective bundles of lines instead of quotients.  As a result, the classes which naturally reflect the positivity of $E$ in our setting are the dual Segre classes $s_{i}(E^{\vee})$ and not the Segre classes $s_{i}(E)$. %\cite[\S2.5]{fulton84} first defines a first Chern class action for line bundles via restricting and taking a meromorphic section.
%Then 
%$$s_i(E^{\vee})\cap\alpha:=\pi_*(c_1^{i+e-1}(\mathcal O_{\mathbb{P}(E)}(1))\cap \pi^*\alpha)\in A_{k-i}(X)$$
%for any $\alpha\in A_k(X)$, 
%where $\pi:\mathbb P(E)\to X$ is the projection map of the projective bundle of quotients associated to $E$, 
%where $\mathcal O_E(1)$ is the relative Serre line bundle, 
%and %where
%$\pi^*$ is the flat pullback of \cite[\S1.7]{fulton84}. 
%(The dual of $E$ appears because \cite{fulton84} uses projective bundles of lines.)

The Chern classes of $E$ are the linear operators defined so that the total Chern class $c(E):=c_0(E)+c_1(E)+\ldots$ is a formal inverse of $s_0(E)+s_1(E)+\ldots$.
%For example $s_0(E)=c_0(E)=1$ and $c_1(E)=c_1(\det E)=s_1(E^{\vee})=-s_1(E)$.
 Since the Chern class operations are commutative and associative (see \cite[\S 3.2]{fulton84}), there is a natural way of defining $P(E_I)\cap [Z]$ for any finite collection of vector bundles $\{ E_{i} \}_{i \in I}$ on $X$, where $P(E_I)$ denotes a weighted homogeneous polynomial on the Chern classes of these bundles (assigning weight $j$ to a $j$th chern class).

We use the following properties of Chern/Segre classes:
\begin{itemize}
\item (Projection formula, cf. \cite[Theorem 3.2.(b,c)]{fulton84}) If $\pi:X\to Y$ is a proper morphism and $E_I$ a collection of vector bundles
on $Y$, then $\pi_*(P(\pi^*E_I)\cap\alpha)=P(E_I)\cap\pi_*\alpha$.
\item (Compatibility with pullback, cf. \cite[Theorem 3.2.(b,d), Proposition 6.3, Example 8.1.6]{fulton84}) If $\pi:X\to Y$ is flat, l.c.i, or if $Y$ is smooth,
then $P(\pi^*E_I)\cap\pi^*\alpha=\pi^*(P(E_I)\cap\alpha)$, where $\pi^*\alpha$ is the flat pullback \cite[\S1.7]{fulton84}, the Gysin action \cite[\S6.2]{fulton84}, or the refined intersection \cite[\S8.1]{fulton84}
respectively. 
\item (Whitney formula, cf. \cite[Theorem 3.2.(e)]{fulton84}) If $0\to E\to F\to G\to 0$ is a short exact sequence of vector bundles, then $c(F)=c(E)\cdot c(G)$.
\item (cf. \cite[Example 3.1.1]{fulton84}) If $L$ is a line bundle, then $$s_i(E\otimes L)=\sum_{j=0}^i(-1)^{i-j}{{e+i}\choose{e+j}}s_j(E)c_1(L)^{i-j}.$$
\item Any weighted homogeneous polynomial in chern classes of vector bundles $P(E_{I})$ can be rewritten as a weighted homogeneous polynomial $Q(E_{I})$ in dual Segre classes of these same vector bundles.  This is easily proved by induction using the formal inverse relationship and $c_{i}(E^{\vee}) = (-1)^{i}c_{i}(E)$.
\end{itemize}

\subsubsection{Numerical equivalence}
We will work with an equivalence relation coarser than rational equivalence. \cite[\S 19]{fulton84} defines a $k$-cycle $Z$ to be \textit{numerically trivial} if
\begin{equation}\label{eq:numtriv}\deg(P(E_I)\cap Z)=0\end{equation} 
for any weight $k$ homogeneous polynomial $P(E_I)$ in Chern classes of a finite set of vector bundles on $X$. Here $\deg:A_0(X)\to\mathbb Z$ is the group morphism that sends any point to 1. %, and $P(E_I)\cap Z\in A_0(X)$ is defined as in \cite[Chapter 3]{fulton84}. 
The quotient of $Z_{k}(X)$ by the numerically trivial cycles is denoted $N_k(X)_{\mathbb Z}$; this is a free abelian group of finite rank by \cite[Example 19.1.4]{fulton84}. It is a lattice inside $N_k(X)_{\mathbb Q}:=N_k(X)_{\mathbb Z}\otimes_{\mathbb Z}\mathbb Q$ and inside
$$N_k(X):=N_k(X)_{\mathbb Z}\otimes_{\mathbb Z}\mathbb R.$$
We call the latter the \textit{numerical group}. It is a finite dimensional real vector space, and its dimension is positive only when $0\leq k\leq\dim X$. If $Z$ is a real $k$-cycle, its class in $N_k(X)$ is denoted $[Z]$.

It is useful to consider the abstract dual notions $N^k(X)_{\mathbb Z}$, $N^k(X)_{\mathbb Q}$, and $N^k(X)$ of $N_k(X)_{\mathbb Z}$, $N_k(X)_{\mathbb Q}$, and $N_k(X)$ with coefficients $\mathbb Z$, $\mathbb Q$, and $\mathbb R$ respectively. We call $N^k(X)$ the \textit{numerical dual group}.  Note that if $P=P(E_I)$ is a weight-$k$ homogeneous polynomial in Chern classes of a finite set of vector bundles, then $P$ induces an element $[P]$ of $N^k(X)$ via the operational nature of Chern classes.  In fact, %looking at Chern polynomials as elements of ${\rm End}(N_*(X))$
we have the formal identification
\begin{equation}\label{eq:numcocyles}N^k(X)=\frac{\mbox{Homogeneous Chern } \mathbb{R}\mbox{-polynomials } P\mbox{ of weight } k}
{\mbox{Chern polynomials }P\mbox{ such that } P\cap\alpha=0\mbox{ for all }\alpha\in N_k(X)}.\end{equation}
%We have similar statements for $\mathbb Z$ and $\mathbb Q$. 

\begin{exmple} $N^1(X)$ is the N\' eron-Severi space of real Cartier divisors modulo numerical equivalence.  Indeed, using the determinant construction, we see that a first Chern class of a rank $r$ locally free sheaf can also be interpreted as a first Chern class of some invertible sheaf.  Thus, $N_{1}(X)$ is the space of curves with $\mathbb{R}$-coefficients modulo classes which have vanishing intersection against first Chern classes of invertible sheaves.  The formal dual $N^{1}(X)$ is then the real space of first Chern classes of invertible sheaves modulo those with vanishing intersection against every curves, and we interpret an invertible sheaf as a rational equivalence class of Cartier divisors.
\end{exmple}

\begin{rmk}In his senior thesis at Princeton, Conner Jager checked that if $X$ is smooth and projective, then $N^k(X)$ (and in fact even $N^k(X)_{\mathbb Q}$) is generated linearly 
by the Chern classes $c_k(E)$ as $E$ ranges through the vector bundles on $X$. (The idea is to use a 
Riemann--Roch isomorphism to show that $N^k(X)$ is generated additively by the Chern 
character classes ${\rm ch}_k(E)$, then show that $c_k(E)$ is in the linear span of ${\rm ch}_k(E^{\oplus s})$ as $s$ ranges among the positive integers.)
\end{rmk} 

\begin{exmple}If $X$ is a projective variety of dimension $n$, then $Z_n(X)=A_n(X)=N_n(X)_{\mathbb Z}=\mathbb Z\cdot[X]$. The morphism $\deg:Z_0(X)\to \mathbb Z$ that sends all points to 1 factors through an isomorphism $\deg: N_0(X)_{\mathbb Z}\to\mathbb Z$.
\end{exmple}

\begin{rmk}\label{rmk:chowattributes}The quotient map $Z_k(X)\to N_k(X)_{\mathbb Z}$ factors through $A_k(X)$. 
Using \eqref{eq:numtriv} we deduce that many of the attributes of Chow groups descend to numerical groups with their natural grading:
\begin{itemize}

\item Proper pushforwards $\pi_*$. % and pullbacks $\pi^*$ from a nonsingular base. 
Dually, the groups $N^k(X)$ have proper pullbacks $\pi^*:=(\pi_*)^{\vee}$.

\item Actions of polynomials in Chern classes for vector bundles: a weighted homogeneous polynomial $P=P(E_{I})$ of degree $i$ maps $N_k(X) \to N_{k-i}(X)$.  We denote the image of $\alpha \in N_{k}(X)$ by $P\cap \alpha$.

\item The projection formula: If $\pi:Y\to X$ is a proper morphism, and $P(E_I)$ is a 
polynomial in the Chern classes on $X$, then for any 
$\alpha\in N_k(X)$,
$$\pi_*(P(\pi^*E_I)\cap \alpha)=P(E_I)\cap \pi_*\alpha.$$

\item Gysin homomorphisms: Suppose that $\pi: Y \to X$ is an l.c.i.~morphism of codimension $d$.  Then \cite[Example 19.2.3]{fulton84} shows that the Gysin homomorphism $\pi^{*}: A_{k}(X) \to A_{k-d}(Y)$ descends to numerical groups. Similarly, $\pi^*$ exists when $\pi:Y\to X$ is a morphism of projective varieties with $X$ smooth. %In particular, $f$ defines an element of $N^{d}(X)$: the homomorphism sending $\alpha \in N_{d}(X)$ to $\deg(f^{*}\alpha)$.
\end{itemize}
\end{rmk}

\begin{rmk}\label{rmk:dualring}Multiplication of polynomials induces a graded ring structure on $N^*(X)$. We call this the \textit{numerical dual ring} of $X$.
If $\pi:Y\to X$ is a proper morphism, then $\pi^*:=(\pi_*)^{\vee}$ is a ring homomorphism by the projection formula.
\end{rmk}

\begin{notn}
Where there is little danger of confusion, we often use $\cdot$ instead of $\cap$ to denote the intersection of cycles with Chern classes or dual classes.
\end{notn}

\begin{caut}
We do not know if the flat Chow pullbacks (\cite[\S1.7]{fulton84}) respect numerical equivalence.  However, flat numerical pullbacks exist when the base is smooth; see Remark \ref{rmk:flatpullbacks}.
\end{caut}

The association $[P]\to P\cap[X]$ induces a natural map 
\begin{equation}\label{eq:cyclification} \varphi:N^{n-k}(X)\to N_k(X),\end{equation}
which is not usually an isomorphism. Its dual is the corresponding natural map $\varphi:N^k(X)\to N_{n-k}(X)$. We have similar statements for $\mathbb Q$-coefficients.

\begin{exmple}If $X$ is projective of dimension $n$, then the map $\varphi: N^1(X)\to N_{n-1}(X)$ is the
numerical version of the cycle map from Cartier divisors to Weil divisors. It is a consequence of \cite[Example 19.3.3]{fulton84} that this map is injective and the dual $\varphi:N^{n-1}(X)\to N_1(X)$ is surjective. (An element in the kernel is $[c_1(\mathcal L)]$ for some line bundle $\mathcal L$ such that $c_1(\mathcal L)\cap[X]$ is numerically trivial in the sense of \eqref{eq:numtriv}. In particular, $\deg(c_1(\mathcal L)\cdot c_1^{n-1}(\mathcal O_X(H))\cap[X])=\deg(c_1^2(\mathcal L)\cdot c_1^{n-2}(\mathcal O_X(H))\cap[X])=0$ for some ample bundle $\mathcal O(H)$ on $X$. The cited reference implies that $c_1(\mathcal L)\cap [C]=0$ for any $1$-cycle $C$. Therefore $[c_1(\mathcal L)]=0$ in $N^1(X)$.)
%We have $N^1(X)=NS(X)\otimes\mathbb R$, where $NS(X)$ is the N\' eron--Severi group of $X$.
%In general, the quotient map from $NS(X)$ to $N^1(X)_{\mathbb Z}$ may have a finite kernel.\qed
\end{exmple}

\begin{exmple}\label{ex:numsing} When $X$ is singular, quite often $\varphi$ is not an isomorphism.  For example, let $Y\subset\mathbb P^N$ be a projective variety of dimension $n$ with $\dim N_{n-1}(Y)>1$. Denote by $X:=\mathcal C(Y)\subset\mathbb P^{N+1}$ the projective cone over $Y$ of dimension $n+1$, and by $\pi:Z\to X$ the blow-up of the vertex, such that $Z$ has the structure of a projective bundle of relative dimension $1$ over $Y$ with bundle map $f:Z\to Y$. 
We claim that $\dim N_1(X)=1$ and $\dim N_n(X)=\dim N_{n-1}(Y)>1$. Therefore $\varphi:N^1(X)\to N_n(X)$ is not an isomorphism. (We verify that $\pi_*f^*$ induces an isomorphism $N_{k-1}(Y)\simeq N_{k}(X)$ for all $k>0$. Here $f^*$ is a smooth pullback, so it respects numerical equivalence. The variety $Z$ contains two notable disjoint Cartier divisors: the zero section $E$ of the geometric vector bundle associated to $\mathcal O_{\mathbb P^N}(1)|_Y$, and the compactifying hyperplane at infinity $F$. Note that $E$ and $F$ are both sections of $f$, isomorphic to $Y$ via $f|_E$ and $f|_F$ respectively, and $\pi(E)$ is the vertex of $X$, while $\pi|_F$ is the identity of $Y$. In particular $X$ contains a copy of $Y$, the intersection $\mathcal C(Y)\cap \mathbb P^N\subset \mathbb P^{N+1}$. We have $N_k(Z)=f^*N_{k-1}(Y)\oplus E\cdot f^*N_k(Y)$. Since $E$ is contracted to a point, $\pi_*(E\cdot f^*N_k(Y))=0$ for all $k>0$. On the other hand, for any $\alpha\in N_{k-1}(Y)$, we have $\pi_*f^*\alpha=0$ if and only if $\alpha=0$.
Indeed  by \cite[Theorem 6.2.(a)]{fulton84} we have $(\pi_*f^*\alpha)|_Y=(\pi|_F)_*(f^*\alpha|_F)=\alpha$.)\qed
\end{exmple}

\begin{rmk}When $X$ is \emph{smooth} and projective, the intersection theory of \cite[Chapter 8]{fulton84} endows $A_*(X)$ with a ring structure, graded by codimension. This descends to $N_*(X)_{\mathbb Z}$. 

By \cite[Example 15.2.16.(b)]{fulton84}, we have an isomorphism ${\rm ch}:K(X)\otimes{\mathbb Q}\to A(X)\otimes{\mathbb Q}$.
A consequence is that any Chow $\mathbb Q$-class is of the form $P(E_I)\cap [X]$ for some Chern polynomial $P(E_I)$ with rational coefficients. Descending to numerical equivalence, the natural morphism $\varphi:N^{n-k}(X)_{\mathbb Q}\to N_k(X)_{\mathbb Q}$ defined above is an isomorphism. The pairing between $N_k(X)_{\mathbb Q}$ and $N^k(X)_{\mathbb Q}$ induced by the ring structure is the same as their pairing as dual spaces. The analogous identification is also valid for $\mathbb R$-coefficients.
For any morphism $\pi:Y\to X$
of relative dimension $d$ from a projective scheme $Y$, we define the pullback $\pi^*:N_k(X)\to N_{k+d}(Y)$
as $\varphi\circ(\pi_*)^{\vee}$.  Note that this pullback agrees with the refined Gysin homomorphisms of \cite[Chapter 8]{fulton84}.

In particular, we recover the classical definition for numerical triviality on smooth varieties: $Z\in Z_k(X)$ is numerically trivial, if $[Z]\cdot\beta=0$ for any $\beta\in N_{n-k}(X)_{\mathbb Z}$.
\end{rmk}

\begin{rmk} \label{rmk:flatpullbacks}
Let $\pi: Y \to X$ be a flat morphism of projective varieties with $X$ smooth.  By \cite[Proposition 8.1.2]{fulton84} for any cycle $Z \in A_{k}(X)$ we have $\pi^{*}[Z] = [\pi^{-1}Z]$ (where $\pi^{-1}$ denotes the flat pull-back of cycles \cite[\S1.7]{fulton84}).
\end{rmk}

\subsection{The pseudoeffective cone}

We say that a class $\alpha\in N_k(X)$ is \textit{effective} if $\alpha=[Z]$ for some effective cycle $Z$. This notion is closed under positive linear combinations, hence it is natural to consider the following:

\begin{defn}\label{def:eff}The closure of the convex cone generated by effective $k$-cycles on $X$ in $N_k(X)$ is denoted $\Eff_k(X)$. It is called the \textit{pseudoeffective} cone.  A class $\alpha\in N_k(X)$ is called \textit{pseudoeffective} (resp. \textit{big}) if it belongs to $\Eff_k(X)$ (resp. to its interior). For classes $\alpha,\beta \in N_{k}(X)$, we use the notation $\alpha \preceq \beta$ to denote that $\beta-\alpha$ is pseudoeffective.

We say that $\beta\in N^k(X)$ is \textit{pseudoeffective} if $\varphi(\beta)\in\Eff_{n-k}(X)$, where $\varphi$ is the map of \eqref{eq:cyclification}.  The pseudoeffective dual classes form a closed cone in $N^k(X)$ that we denote $\Eff^k(X)$.
\end{defn}

The pseudoeffective cone is full-dimensional. In Corollary \ref{cor:effsalient} we show that it is also salient.

\begin{lem}\label{lem:ci big}If $h_1,\ldots,h_k$ are ample classes in $N^1(X)$, then 
$h_1\cdot\ldots\cdot h_k\cap[X]$ is big.
\end{lem}
\begin{proof}  It suffices to consider the case when each $h_{i} \in N^{1}(X)_{\mathbb{Q}}$. If $Z_j$ are distinct subvarieties whose classes generate $N_{n-k}(X)$, then $\alpha=\sum_j[Z_j]$ is big. There exists an integer $m\gg 0$ and ample Cartier divisors $D_1,\ldots, D_k$ of class $mh_1,\ldots,mh_k$ respectively such that each $D_i$ contains $\cup_jZ_j$ in its Weil support and the set $\cap_iD_i$ is of dimension $n-k$. Then 
$m^kh_1\cdot\ldots\cdot h_k\cap[X]=\alpha+\beta$ for some effective class $\beta$.
The sum of a big and a pseudoeffective class is big.
\end{proof}

(Pseudo)effectivity is the natural covariant positivity notion for cycles; it is preserved under proper pushforward and flat pullback from a smooth target. We can say more
when dealing with a dominant morphism.

\begin{rmk}\label{rmk:dominant eff}Let $\pi:Y\to X$ be a dominant morphism of projective varieties.
If $Z\subset X$ is an arbitrary closed subvariety of dimension $k$, then there exists
an effective class $\alpha\in N_k(Y)_{\mathbb Q}$ such that
$\pi_*\alpha=[Z]$. (Let $h$ be an ample divisor class in $N^1(Y)_{\mathbb Z}$ 
and let $T$ be an irreducible component of $\pi^{-1}Z$ that dominates $Z$.
Let $d$ be the relative dimension of the induced morphism $\pi:T\to Z$.
Then $h^d\cdot[T]$ is an effective $\mathbb Q$-cycle and 
$\pi_*(h^d\cdot[T])=c[Z]$ for some $c\in\mathbb Q_{+}$. 
Put $\alpha=\frac 1c h^d\cdot[T]$.)

A consequence is that $\pi_*\Eff_k(Y)$ has dense image in $\Eff_k(X)$. 
Note that it is possible for the image of a closed convex cone under a linear
map of finite dimensional real vector spaces to no longer be closed.
We nonetheless prove that $\pi_*\Eff_k(Y)=\Eff_k(X)$ in Corollary \ref{cor:surjeff}.
\end{rmk}

\subsection{The nef cone}

The cone dual to $\Eff_k(X)$ in $N^k(X)$ is the \textit{nef} cone $\Nef^k(X)$. Any element $\beta\in\Nef^k(X)$ is called \textit{nef}. Attesting to its
contravariant nature, nefness is preserved under proper pullbacks.

\begin{rmk}Since $\Eff_k(X)$ is full-dimensional, $\Nef^k(X)$ is salient.
In Lemma \ref{lem:pliantupsef} we see that this cone is also full-dimensional.
\end{rmk}

\begin{rmk}Let $\pi:Y\to X$ be a dominant morphism of projective varieties.
If $\beta\in N^k(X)$ is such that $\pi^*\beta\in\Nef^k(Y)$, then $\beta\in\Nef^k(X)$.
(Nefness on $X$ is verified by testing nonnegative pairing against the closed
subvarieties of dimension $k$. Then apply the projection formula and Remark \ref{rmk:dominant eff}.)
\end{rmk}

\begin{exmple}If $h_1,\ldots,h_k\in\Nef^1(X)$, then $h_1\cdot\ldots\cdot h_k\in\Nef^k(X)$.  In Corollary \ref{cor:ci nef} we see that if $h_i$
are ample for all $i$, then their intersection is in the interior of $\Nef^k(X)$. Note that complete intersections do not always generate a full-dimensional cone.
For example when $X=G(2,4)$, then $\dim N^1(X)=1$ while $\dim N_2(X)=2$.

More generally, if $E$ is a nef vector bundle (i.e. $\mathcal O_{\mathbb P(E)}(1)$ is a nef line bundle), then $c_k(E)\in\Nef^k(X)$ for all $k$. (This is an easy consequence of \cite[Example 12.1.7.(c)]{fulton84}.)  If $E$ is ample, then $c_k(E)$ belongs to the strict interior of $\Nef^k(X)$  for all $k$.\qed
\end{exmple}

For $k>1$, \cite{delv11} provides examples of nef classes that do not have nef intersection in $N^*(X)$, and of nef classes that are not pseudoeffective.  There, $X$ is a self-product of an abelian variety. In the next section we present a more geometric positivity notion that avoids such pathologies.

\section{The pliant cone}\label{s:pliant}

Just as nef divisors (considered as as limits of semiample divisors) are modeled after the hyperplane class on $\mathbb P^n$, pliant classes are modeled after Schubert cycle classes on
Grassmannians up to taking products.  

\begin{defn}(\cite[\S14.5]{fulton84})
Fix positive integers $k$ and $e$.  Let $\lambda = (\lambda_{1},\ldots,\lambda_{k})$ be a decreasing partition of $k$ involving only non-negative integers that are no greater than $e$.  The (weighted) Schur polynomial $s_{\lambda}$ is defined to be the determinant in formal variables $c_{1},\ldots,c_{e}$
\begin{equation*}
s_{\lambda} := \left| \begin{array}{cccc} c_{\lambda_{1}} & c_{\lambda_{1}+1} & \ldots & c_{\lambda_{1}+k-1} \\ c_{\lambda_{2} - 1} & c_{\lambda_{2}} & \ldots & c_{\lambda_{2}+k-2} \\
\vdots & \vdots & \ddots & \vdots \\ c_{\lambda_{k}-k+1} & c_{\lambda_{k}-k+2} & \ldots & c_{\lambda_{k}} \end{array} \right|
\end{equation*}
where by convention $c_{0} = 1$ and $c_{i} = 0$ if $i \not \in [0,e]$.  If we assign the weight $i$ to the variable $c_{i}$, then $s_{\lambda}$ is a degree $k$ weighted-homogeneous polynomial.  Given a vector bundle $E$ of rank $e$, then $s_{\lambda}(E)$ denotes the corresponding Schur polynomial in the Chern classes of $E$. The Chern classes $c_k(E)$ and the dual Segre classes $s_k(E^{\vee})$ are particular cases of this construction.
\end{defn} 

Let $Q$ denote the tautological quotient bundle over a Grassmannian $\mathbb{G}$. Note that $Q$ is globally generated. The Schubert cycles on $\mathbb{G}$ have numerical class given by Schur polynomials in the Chern classes of $Q$ (see \cite[\S 14.6]{fulton84} or \cite[Remark 8.3.6]{lazarsfeld04}). When $E$ is a globally generated vector bundle on $X$, then the Schur polynomial classes of $E$ are pullbacks of Schubert cycle classes via the induced Gauss map. See also Example \ref{ex:grassmann}.

\begin{defn}\label{def:pliant}The \textit{pliant cone} $\pl^k(X)$ is the closed convex cone in $N^k(X)$ 
generated by monomials $\prod_is_{\lambda_i}(E_i)$ in Schur polynomial classes %(see \cite[Example 8.3.7]{lazarsfeld04})
of globally generated vector bundles on $X$. 
\end{defn}

The construction is also motivated by the work of Fulton--Lazarsfeld \cite{fl83} on positive polynomials of ample vector bundles. We will see that the pliant cone satisfies the desirable intersection-theoretic properties described in the introduction.

\begin{exmple}
For any globally generated vector bundle $E$ we have that $c_{1}(E)$ is the class of the nef line bundle $\det E$.  Thus $\pl^{1}(X)=\Nef^1(X)$.
\end{exmple}

\begin{rmk} Note that the definition of the pliant cone is stable under products. In particular, the class of a complete intersection is pliant. We do not know if the pliant cone coincides with the cone generated by Schur polynomial classes of globally generated bundles (without taking products).
\end{rmk}

\begin{lem}\label{lem:pl generates}The pliant cone $\pl^k(X)$ generates $N^k(X)$ as a vector space, i.e. it is full-dimensional.
\end{lem}
\begin{proof}Since the pliant cone is closed under products, it suffices to show that for any vector bundle $E$ on $X$ the Chern class $c_{i}(E)$ can be expressed as a sum of products of Chern classes of globally generated vector bundles.  (Note that the Chern class $c_{k}$ is the Schur polynomial corresponding to the partition $(k,0,\ldots,0)$.)

The proof is by induction on $i$.  Let $H$ be a fixed very ample divisor on $X$.
There exists a positive integer $m$ such that $E(mH)$ is globally generated.  The formula for chern classes of a tensor product (as recalled in Section \ref{charclasssec}) expresses $c_{i}(E(mH))$ as a sum of $c_{i}(E)$ with other terms involving $c_{1}(H)$ and $c_{j}(E)$ for $j<i$.  By induction on $i$, we conclude that $c_{i}(E)$ can be written as a linear combination of products of Chern classes of globally generated vector bundles.
\end{proof}

\begin{rmk}\label{rmk:plpull}Since global generation is preserved by the pullback of vector bundles, 
if $\pi:Y\to X$ is a morphism of projective varieties, then $\pi^*\pl^k(X)\subset\pl^k(Y)$.
\end{rmk}

\begin{lem}\label{lem:pliantupsef}If $Z\subset X$ is a subvariety of dimension $d$, then for any $[P]\in\pl^k(X)$, we have $P\cap[Z]\in\Eff_{d-k}(X)$. In particular, we have an inclusion $\pl^k(X)\subset\Eff^k(X)\cap\Nef^k(X)$ so that %The cone $\Eff^k(X)$ was defined in \ref{def:eff}. Moreover
$\pl^k(X)$ is a salient cone and $\Nef^k(X)$ is full-dimensional.\end{lem}
\begin{proof}This follows from \cite[Example 12.1.7.(a)]{fulton84}.
\end{proof}

\begin{cor}\label{cor:effsalient}The cone $\Eff_k(X)$ is salient.\end{cor}
\begin{proof} The dual of a full-dimensional cone is salient.\end{proof}

\begin{rmk}The previous result seems to have been missing from the literature.  Over $\mathbb C$, it is implied by \cite[Lemma 2.2]{djv13}. Similar statements are proved in \cite[Proposition 1.3]{bfj09} and \cite[Lemma 2.3]{chms13} for Cartier divisors.
\end{rmk}

\begin{exmple}\label{exmple:homogeneouspos}If $H$ is a projective nonsingular homogeneous space, then
$\pl^k(H)\subseteq\Eff^k(H)\subseteq\Nef^k(H)$ for all $k$.
(Lemma \ref{lem:pliantupsef} gives the first inclusion. For the second inclusion,
note that using the group action and Kleiman's Lemma (\cite[2. Theorem.(i)]{kleiman74}) we can deform any two subvarieties $V$ and $W$ of $H$ until they meet properly, hence $V\cdot W$ is algebraically equivalent to an effective cycle.) When $H$ is one of the  examples of abelian varieties in \cite{delv11}, the last inclusion is strict.\qed
\end{exmple}

\begin{exmple} \label{ex:plandflop}
On $\mathbb{P}^{1}$ consider the vector bundle $\mathcal{E} = \mathcal{O} \oplus \mathcal{O} \oplus \mathcal{O}(-1)$.  Set $X = \mathbb{P}(\mathcal{E})$ and let $\pi: X \to \mathbb{P}^{1}$ denote the projection.  Let $\xi$ denote the class of $\mathcal{O}_{\mathbb P(\mathcal E)}(1)$, and let $f$ denote the class of a fiber of $\pi$.  The Grothendieck relation is $\xi^{3} = -\xi^{2}f = -1$.  Using for example \cite[Theorem 1.1]{ful11} and \cite[Proposition 7.1]{fl13z}, we find that
\begin{equation*}
\Eff^{1}(X) = \Mov^{1}(X) = \langle f, \xi \rangle \qquad \qquad \Nef^{1}(X) = \langle f, \xi + f \rangle
\end{equation*}
and
\begin{equation*}
\Eff^{2}(X) = \langle \xi f, \xi^{2} \rangle \qquad \qquad \Nef^{2}(X) = \langle \xi f, \xi^{2} + \xi f \rangle.
\end{equation*}

\noindent We prove that $\pl^2(X)=\Nef^2(X)$. Since $\xi f = (\xi + f)f$ is a product of nef divisors, it is enough to show that $\xi^2+\xi f$ is pliant. Consider the bundle $Q$ given by the short exact sequence
$$0\to\mathcal O_{\mathbb P(\mathcal E)}(-1)\to\pi^*\mathcal E^{\vee}\to Q\to 0.$$
Then $Q$ is globally generated and $c_2(Q)=\xi^2+\xi f$. 
\qed
\end{exmple}

\begin{rmk}To compute the pliant cone we often guess that it coincides with one of the other positive cones and then construct globally generated vector bundles with specified
Schur (often Chern) classes. For the other positive cones we usually have better techniques. A telling example is that of projective bundles over curves of arbitrary genus (see \cite[Theorem 1.1]{ful11} for $\Eff_k(X)$, \cite[Proposition 7.1]{fl13z} for $\Mov_k(X)$, and Example \ref{ex:projbundleupsef} for $\upsef^k(X)$). In Example \ref{ex:projbundlebpf} we compute all the $\bpf^k(X)$ cones when $X$ is a projective bundle over $\mathbb P^1$.
\end{rmk}

\begin{exmple}\label{ex:grassmann}
If $X$ is a product of Grassmann varieties, then
$$\pl^{k}(X)=\Eff^k(X)=\Nef^{k}(X).$$ 
These cones are rational polyhedral, generated by classes of products of Schubert cycles from each Grassmann factor. In particular Definition \ref{def:pliant} agrees with the definition given in the introduction in terms of products of pullbacks of effective classes on Grassmann varieties. (Consider first the case where $X=\mathbb G$ is a single Grassmann variety. Then $N^k(\mathbb G)$ and $N_k(\mathbb G)$ admit dual bases of effective classes determined by the Schubert cycles (see \cite[\S 14.6]{fulton84} or \cite[Remark 8.3.6]{lazarsfeld04}). These are Schur classes of the  universal quotient bundle on $\mathbb G$, which is a globally generated vector bundle. When $X$ is a product, then the classes of the product of Schubert cycles in each Grassmann factor give bases of $N_k(X)$ for all $k$ by \cite[Proposition 14.6.5]{fulton84}. It is straightforward to check that the bases are dual to each other, and that they are pliant.) \qed
\end{exmple}

We show that complete intersections are in the interior of the pliant cone
and describe several important applications of this result.

\begin{lem}\label{lem:ci intpliant}If $h_1,\ldots,h_k$ are ample classes in $N^1(X)$, then $h_1\cdot\ldots\cdot h_k$ is in the interior of $\pl^k(X)$.
\end{lem}

\begin{proof}Let $h$ be any ample class in $N^1(X)$. There exists $m\gg 0$ such that
$mh_i-h$ is ample for all $i$. Then $m^kh_1\cdot\ldots\cdot h_k=h^k+P$, where $[P]\in\pl^k(X)$. Therefore it is enough to show that $h^k$ is in the interior of $\pl^k(X)$
for some ample class $h$.

Lemma \ref{lem:pl generates} and its proof allow us to choose finitely many monomials in Chern classes of finitely many \textit{ample} globally generated vector bundles $\{ E_{i} \}_{i \in I}$ on $X$ 
such that these monomials span $N^k(X)$ as a vector space. 
The sum of all these monomials is a polynomial with
positive coefficients $P(E_I)$ whose class necessarily lies in the interior of $\pl^k(X)$.

If $E:=\oplus_{i\in I}E_I$, then $c(E)=\prod_ic(E_i)$, where $c(E)=1+c_1(E)+c_2(E)+\ldots$ is the total Chern class of $E$. In particular, for all $j$ and for all $i\in I$, we have 
\begin{equation}\label{eq:Chern upgrade}c_j(E)=c_j(E_i)+P_{ij}(E_I)\end{equation}
 for some $[P_{ij}(E_I)]\in\pl^j(X)$. Note that $E$ is again globally generated and ample. It is important to work with Chern classes here, instead of arbitrary Schur classes, because this ensures that the $P_{ij}$'s have no negative coefficients.

Let $R(E)$ be the polynomial obtained from $P(E_I)$ by replacing every occurrence of $c_j(E_i)$ by $c_j(E)$. By \eqref{eq:Chern upgrade}, we can write $R(E)=P(E_I)+P'(E_I)$ where $[P'(E_I)]\in\pl^k(X)$,
hence $[R(E)]$ is also in the interior of $\pl^k(X)$.

Let $\gamma:X\to\mathbb G$ be the Gauss map induced by $E$. 
Then $[R(E)]=\gamma^*[R(Q)]$, where $Q$ is the universal quotient bundle 
on $\mathbb G$. Let $C=\gamma^*\pl^k(\mathbb G)\subset\pl^k(X)$. 
Since $C$ contains $[R(E)]$, any element in the interior of $C$ is also interior to 
$\pl^k(X)$. 

Since $E$ is ample, $\gamma$ is finite. (If $\gamma$ contracts a curve $C$, then $E|_C$ is trivial. This contradicts ampleness. See \cite[Proposition 6.1.7]{lazarsfeld04}.) If $a$ is a generator for the cone of ample divisors on $\mathbb G$, then $h=\gamma^*a$ is ample on $X$. 
Lemma \ref{lem:ci big} and Example \ref{ex:grassmann} show that $a^k$ is in the interior of $\pl^k(\mathbb G)$. Then $h^k$ is in the interior of $C$, therefore also in the interior
of $\pl^k(X)$.
\end{proof}

\begin{cor}\label{cor:ci nef}If $h_1,\ldots,h_k$ are ample divisors classes, then $h_1\cdot\ldots\cdot h_k$
is in the interior of $\Nef^k(X)$.
\end{cor}

\subsection{Geometric applications}

\begin{cor}[Geometric norms]\label{cor:norms}If $h$ is an ample divisor class on $X$, then for all $k$
there exists a norm $\|\cdot\|$ on $N_k(X)$ such that $\|\alpha\|=h^k\cap\alpha$
for any $\alpha\in\Eff_k(X)$.
\end{cor}
\begin{proof}By Lemma \ref{lem:pl generates} and Corollary \ref{cor:ci nef}, we can choose $\beta_1,\ldots,\beta_m$ nef dual classes that span $N^k(X)$ and such that
 $[h^k]=\sum_i\beta_i$. Then $\|\cdot\|=\sum_i|\beta_i\cap\cdot|$ is a norm on $N^k(X)$
with the required property.
\end{proof}

\begin{cor}\label{cor:deg0}Let $X$ be a projective variety. If $\alpha\in\Eff_k(X)$ has degree zero with
respect to some polarization $H$ on $X$, i.e. $\deg(c_1^k(\mathcal O_X(H))\cap\alpha)=0$, then $\alpha=0$.
\end{cor}

\begin{cor}[Finiteness of integral classes of bounded degree]Let $X$ be a projective variety, and let $H$ be a very ample divisor on $X$.  Then for all $M>0$, 
$$\#\{\alpha\in N_k(X)_{\mathbb Z}\cap\Eff_k(X)\ |\ \deg(c_1^k(\mathcal O_X(H))\cap\alpha)<M\}<\infty.$$ \end{cor}

\begin{rmk}The last three corollaries were known for curve classes; see \cite[Theorem 1.4.29 and Example 1.4.31]{lazarsfeld04}. When working over $\mathbb C$, the result of Corollary
\ref{cor:deg0} can be improved to homological equivalence. See \cite[Proposition 2.1 and Lemma 2.2]{djv13}. Then Corollary \ref{cor:ci nef} is also valid for homological equivalence on complex projective varieties.
\end{rmk}

Many cohomology theories have the Strong Lefschetz property.
A long standing open question concerning numerical groups is if they verify it
as well.

\begin{conj}[Strong Letschetz] Let $X$ be a smooth projective variety of dimension
$n$. Let $h$ be an ample divisor class. Then $\cap h^{n-2k}:N^k(X)\to N^{n-k}(X)$ is an isomorphism for all $k\leq\lfloor\factor n2\rfloor$.
\end{conj}

Corollary \ref{cor:deg0} shows that we can exclude the pseudoeffective case from the conjecture. Note that the smoothness condition is necessary: If $X$ is singular, then usually $\dim N^{n-1}(X)=\dim N_{n-1}(X)>\dim N^1(X)$ (see Example \ref{ex:numsing}).
\vskip.25cm 
That the degree of a cycle with respect to an arbitrary ample polarization
restricts to a norm on the pseudoeffective cone also allows us
to construct ``bounded" lifts for effective cycles 
by dominant morphisms.

\begin{prop}\label{prop:deg cover bound} Let $\pi: Y \to X$ be a surjective morphism of projective varieties. Let $\|\cdot\|$ and $|\cdot|$ be arbitrary norms on $N_k(Y)$ and $N_k(X)$ respectively. There is some constant $C$ depending only on $\pi$ and on the choice of norms on $N_k(Y)$ and $N_k(X)$ such that for any effective $\mathbb{R}$-$k$-cycle $Z$ on $X$, there is an effective $\mathbb R$-$k$-cycle $Z'$ on $Y$ with $\pi_{*}Z' = Z$ satisfying
\begin{equation*}
\|[Z']\| \leq C|[Z]|.
\end{equation*}
When $Z$ has integer coefficients, we can choose $Z'$ having rational coefficients
with denominators bounded independently of $Z$.
\end{prop}

\begin{proof} By repeating the argument for each component, we can assume
that $Z$ is a closed subvariety of $X$. Let $A$ be a very ample divisor on $Y$ and let $H$ be a very ample divisor on $X$. By Corollary \ref{cor:norms}, we can assume that the restriction of the norms $\|\cdot\|$ and $|\cdot|$ to $\Eff_k(Y)$ and $\Eff_k(X)$ respectively are the degree functions with respect to the polarizations $A$ and $H$ respectively. Let $T$ be a component of a $(\dim X)$-dimensional complete intersection of elements of $|A|$ that dominates $X$.  
Then $\pi_*T=cX$ where $c$ is a positive integer depending only on $\pi$ and on $A$.

We do induction on $\dim X\geq k$. When $\dim X=k$, then $Z=aX$ for some $a\geq 0$ and we can put $Z'=\frac acT$. Put $C=\frac{\|[T]\|}{|[X]|}$.
Now suppose $\dim X>k$. Let $\imath:T\hookrightarrow Y$ be the inclusion.
We can assume that $T=Y$. Indeed $\imath_*$ is continuous and preserves
pseudoeffectivity, and the norm induced by the degree with respect to $A$
restricts to the norm induced by $A|_T$. Therefore we can assume that $\pi$ is generically finite and surjective.

There exists an effective Cartier divisor $E$ on $Y$ such that
$-E$ is $\pi$-ample. Replacing $H$ by a fixed multiple depending only on $\pi$ and $E$, we can assume that $\pi^*H-E$ is ample on $Y$. 
Using Corollary \ref{cor:norms}, the ample divisors $\pi^*H-E$ and $A$ determine equivalent norms on $N_k(Y)$, so without loss of generality we can assume that $A=\pi^*H-E$. 
By abuse we also use $E$ as notation for the support of $E$.

Let $\pi': Y' \to X'$ denote a flattening of $\pi$.  Let $S \subset X$ denote the union of $\pi(E)$ and the exceptional locus for the birational morphism $X' \to X$. Note that $\pi|_E:E\to S$ and the restrictions $A|_E$ and $H|_S$ only depend on $\pi$ and on $A$ and $H$. Also note that $\dim S_{i}<\dim X$ for every component $S_{i}$ of $S$.

By applying induction to the components $S_{i}$ of $S$ and the maps $\pi|_{S_{i}}$, we see that the conclusion holds if $Z\subset S$.  If $Z$ is not contained in $S$,
let $\bar Z$ be a $k$-dimensional component of $\pi^{-1}\{Z\}$ that dominates $Z$. Then $\pi_*\bar Z=c'Z$ where $c'>0$ and $\bar Z$ is irreducible, not contained
in $E$.  Furthermore, by taking strict transforms of $Z$ and $\bar Z$ on $X'$ and $Y'$ respectively,  \cite[Example 1.7.4]{fulton84} shows that $c' \leq \deg(\pi')$ for the flat map $\pi'$.  The function
$$t\to (\pi^*H-tE)^k\cap[\bar Z]$$
is decreasing on $[0,1]$. This and the projection formula imply
$$A^k\cap[\bar Z]=(\pi^*H-tE)^k\cap[\bar Z]\leq\pi^*H^k\cap[\bar Z]=H^k\cap c'[Z].$$
%We show that $c'$ is bounded independently of $Z$. For this, let $\pi':Y'\to X'$
%be a flattening of $\pi$, i.e. a flat birational modification of $\pi$. By induction and the previous argument, we can assume
%that the strict transform of $Z$ in $X'$ is nonempty and meets the flat locus of $\pi'$. Then we can assume
%that $\pi'=\pi$. In this case $\bar Z$ is a component of the effective cycle $\pi^*Z$. By \cite[Example 1.7.4]{fulton84}, we get $c'\leq\deg\pi$. Put $Z'=\frac 1{c'}\bar Z$.

\noindent One can choose the constant $C$ by taking the maximum over $\deg(\pi')$, all constants
showing up in the finitely many induction steps, and all finitely many constants appearing as proportionality bounds between equivalent norms.
Similarly, one obtains the last statement of the proposition by taking a maximum over $\deg(\pi')$ and all constants showing up in the finitely many induction steps. 
\end{proof}

%\vskip.5cm
%\textit{The general case.} The exist blow-ups $f:Y'\to Y$ and $g:X'\to X$ and
%a flat morphism $\pi':Y'\to X'$ such that $\pi\circ f=g\circ \pi'$. We say that
%$\pi'$ flattens $\pi$. As in the birational case, let $E$ be effective on $X'$ such
%that $-E$ is $g$-ample. Choose $A'$ ample on $Y'$ such that $A'-f^*A$ is ample.
%If $Z\subset g(E)$, the result follows by induction.

%If $Z$ is not contained in $g(E)$, let $\bar Z$ be its strict transform on $X'$.
%For $\bar Z$ and $\pi'$, construct $S$ and $t$ as in the flat case
%and put $Z'=\frac 1tf_*[S]$. The ampleness of $A'-f^*A$ implies
%$$A^k\cap[Z']\leq (A')^k\cap\frac 1t[S].$$
%The flat case for $\pi'$ and the birational case for $g$ yield
%a constant $C=C(\pi,A,H)$ such that $(A')^k\cap\frac 1t[S]\leq C\cdot (H^k\cap[Z])$.
%The result follows.

%When $Z$ has integer coefficients, then the coefficients of $Z'$ are rational
%with denominators bounded by the products of all $t$ constructed on flattenings.
%There are finitely many flattenings that may appear from the induction steps,
%and they are all independent of $Z$.
%\end{proof}

\begin{cor}\label{cor:surjeff}If $\pi:Y\to X$ is a dominant morphism
of projective varieties, then $\pi_*:\Eff_k(Y)\to\Eff_k(X)$ is
surjective for all $k$.
\end{cor}
\begin{proof}Let $\alpha$ be a pseudoeffective class on $X$. 
%The effective case was covered already by Remark \ref{rmk:dominant eff}.
Write $\alpha$ as a limit of effective classes $\alpha_i$.
For each $i$, Proposition \ref{prop:deg cover bound} constructs
an effective class $\beta_i$ on $Y$ such that $\pi_*\beta_i=\alpha_i$
whose degree with respect to some polarization on $Y$ is bounded
independently of $i$. Since the degree restricts to a norm on the 
pseudoeffective cone, we can find a limit point $\beta$ for the 
sequence $\beta_i$. Note that $\beta$ is pseudoeffective. Since $\pi_*$ is continuous, $\pi_*\beta=\alpha$.
\end{proof}

We also use Lemma \ref{lem:ci intpliant} to construct bases for $N^k(X)$ with good positivity properties, at least when $X$ is smooth. The second part of the following lemma is an important technical instrument in the proof of \cite[Theorem 8.9]{fl14k}.

\begin{lem}\label{lem:pliantbases}Let $X$ be a smooth projective variety of dimension $n\geq 2$. Then
\begin{enumerate}[i)]
\item $N^k(X)$ is generated by pliant classes with irreducible representatives $\{T_r\}$.
\item If $\pi:X\to Y$ is a surjective morphism to a projective variety with $\dim Y\geq n-k$, then we can arrange such that $T_r$ is not contracted by $\pi$ for any $r$.
\end{enumerate}
\end{lem}

\begin{proof}Arguing as in Section \ref{charclasssec} and Lemma \ref{lem:pl generates}, we can find a set of very ample vector bundles $\{E_i\}$ such that $N^k(X)$ is generated by weight $k$ monomials in dual Segre classes $s_j(E_i^{\vee})$. These monomials belong to $\pl^k(X)$. 

Let $\mathbb P$ denote the fiber product $\times_X\mathbb P(E_i)$, let $\xi_i$ denote the pullback to $\mathbb P$ of the Serre bundle $\mathcal O_{\mathbb P(E_i)}(1)$, and let $p:\mathbb P\to X$ be the (smooth) projection map of relative dimension $d$. 

The proof of \cite[Proposition 3.1.(b)]{fulton84} shows that the weight $k$ dual Segre monomials in the $E_i$ are given by $p_*(\prod_{j=1}^{d+k}\xi_{i_j})$. The number of repetition of each index $i$ in the list of the $i_j$ determines which dual Segre class of $E_i$ appears in the monomial. Since we want to allow several Segre classes of the same bundle to appear in a monomial, we repeat each $E_i$ in the initial list $k$ times so that each class can be obtained from a different factor in $\mathbb P$.

Since $\dim X\geq 2$ and the $E_i$ are very ample vector bundles, the linear systems $|\xi_i|$ are not composites with a pencil for any $i$. Then Bertini's theorem implies that the support of a general complete intersection $\prod_{j=1}^{d+k}\xi_{i_j}$ is irreducible, and then the same is true of its image $T_r$ through $\pi$. 

For part $ii)$, let $h$ be a very ample divisor class on $Y$. Since $T_r$ is effective, it is contracted by $\pi$ if and only if $[T_r]\cdot\pi^*h^{n-k}=0$. The class $\pi^*h^{n-k}$ is effective and nonzero under the assumption $\dim Y\geq n-k$. It is enough to prove that $[T_r]$ belongs to the interior of $\pl^k(X)$. Knowing that pliancy is closed under products, and that complete intersections are interior (cf. Lemma \ref{lem:ci intpliant}), it is enough to check that we can choose $E_i$ such that every nonzero $s_j(E_i^{\vee})$ belongs to the interior of $\pl^j(X)$. For this, replace each $E_i$ by $E_i\otimes\det E_i$ in the initial list. Note that $c_1(\det E_i)=c_1(E_i)=s_1(E_i^{\vee})$. Then the formula in \cite[Example 3.1.1]{fulton84} shows that the linear span of the dual Segre monomials is unchanged. Furthermore $s_j((E_i\otimes\det E_i)^{\vee})$ is a positive linear combination of dual Segre monomials of $E_i$, one of which is a positive scalar multiple of the interior complete intersection class $c_1^j(\det E_i)$.
\end{proof}

\section{Universally pseudoeffective classes}\label{s:upsef}

Universally pseudoeffectivity is the positivity notion that directly
generalizes the intersection theoretic properties of nef divisors.

\begin{defn}\label{defn:upsef}We say that $\alpha\in N^k(X)$ is \textit{universally pseudoeffective} if $\pi^*\alpha\in\Eff^k(Y)$ for any proper morphism
$\pi:Y\to X$ from a projective variety $Y$. The cone of all such is denoted by $\upsef^k(X)$.
\end{defn}

\begin{rmk}\label{rem:upsefnef} Universally pseudoeffective classes are nef.\end{rmk}

\begin{exmple} \label{exm:bdppexample}  For any projective variety $X$ we have $$\upsef^1(X)=\Nef^1(X).$$
(Nefness for divisors is preserved by pullback and nef divisors are pseudoeffective, which implies $\Nef^1(X)\subseteq\upsef^1(X)$. If $\alpha$
is an universally pseudoeffective class of a Cartier divisor, then $\alpha\cap [C]$ is a pseudoeffective 0-cycle for any irreducible curve $C$ in $X$, hence $\alpha$ is a nef divisor class.)\qed
\end{exmple}

\begin{rmk}\label{rmk:upsefgood}Lemma \ref{lem:pliantupsef} shows that $\pl^k(X)\subseteq\upsef^k(X)$ for all $k$. Together with Remark \ref{rem:upsefnef} this implies that $\upsef^k(X)$ is full-dimensional, salient, and contains complete intersections in its strict interior.
\end{rmk}

\begin{exmple}\label{ex:sphericalupsef}If $X$ is a nonsingular projective spherical (e.g. toric) variety, then $\upsef^k(X)=\Nef^k(X)$ for all $k$. (The proof is analogous to \cite[Theorem 3.4]{li13}. Let $\pi:Y\to X$ be a projective morphism, and let $\eta\in\Nef^k(X)$.
Let $\Gamma:Y\to X\times Y$ be the graph morphism associated to $\pi$.
We use the same notation for its image. By \cite[Corollary 3.3]{li13}, drawing on \cite[Lemma 3]{fmss95}, $\Gamma$ is rationally equivalent to 
an effective cycle $\sum_ic_iA_i\times B_i$, where $A_i$ are irreducible
subvarieties of $X$, and $B_i$ are irreducible subvarieties of $Y$. Note that 
$\pi^*\eta=p_{2*}([\Gamma]\cdot p_1^*\eta)$. Then $$\pi^*\eta=\sum_ic_ip_{2*}(p_1^*([A_i]\cdot\eta)\cdot p_2^*[B_i])=\sum_{i, \dim A_i=k}(c_i[A_i]\cdot\eta) \, [B_i]$$ which 
is in fact effective.)\qed
\end{exmple}

\begin{exmple}\label{ex:upsefcurves}If $X$ is a nonsingular projective variety of dimension $n$, and $\Mov_1(X)$ denotes the movable cone of curves,
then $$\upsef^{n-1}(X)=\Mov_1(X).$$
(If $\alpha\in\upsef^{n-1}(X)$, then $\alpha\cap[D]\in\Eff^{n-1}(D)$
for any effective divisor $D$. By \cite{bdpp04} and its extension to arbitrary characteristic in \cite[\S2.2.3]{fl13z}, it follows that $\alpha\in\Mov_1(X)$.

Let now $\alpha\in\Mov_1(X)$ and let $\pi:Y\to X$ be a morphism from a projective variety $Y$ and let $Z=\pi(Y)$ with its closed embedding $\imath:Z\hookrightarrow X$. Write $p$ for the induced morphism $Y\to Z$.
If $\dim Z<n-1$, then $\pi^*\alpha=0$.
If $\dim Z=n-1$, then $\imath^*\alpha\in\Eff^{n-1}(Z)=\upsef^{n-1}(Z)$ since we can write
$\alpha$ as a limit of effective curve cycles without components in $Z$. Therefore
$\pi^*\alpha\in p^*\upsef^{n-1}(Z)\subset\Eff^{n-1}(Y)$. 

Finally, suppose $\pi$ is dominant.  Let $\pi':Y'\to X'$ be a flat birational model of $\pi$; up to base change over an alteration (\cite{dejong96}), we can assume that $X'$ is smooth.  Note that the pullback $\alpha'$ of $\alpha$ to $X'$ is a movable curve by the projection formula and the main result of \cite{bdpp04}. Then $(\pi')^*(\alpha'\cap[X'])=\varphi\circ(\pi'_*)^{\vee}(\alpha')$
is pseudoeffective, because flat pullbacks preserve effectivity for cycles.  Thus the pushforward $\pi^{*}\alpha \cap [Y]$ is also pseudoeffective. ) \qed
\end{exmple}

\begin{exmple}If $X$ satisfies $\Eff^k(X)={\rm S}^k\Nef^1(X)$, where ${\rm S}^k\Nef^1(X)$ is the cone in $N^k(X)$ generated by complete intersections, then $${\rm S}^k\Nef^1(X)=\pl^k(X)=\upsef^k(X)=\Eff^k(X).$$ This is the case for example when $X=A\times A$, and $A$ is a very general complex abelian surface, or when $X=E^n$, where $E$ is a complex elliptic curve with complex multiplication (cf. \cite{delv11}).
\end{exmple}

\begin{prop} \label{prop:domupsef} Let $\pi:Y\to X$ be a dominant morphism of projective varieties.
If $\pi^*\alpha\in\upsef^k(Y)$ for some $\alpha\in N^k(X)$, then $\alpha\in\upsef^k(X)$.
\end{prop}
\begin{proof} Let $Z\to X$ be a morphism and let $T$ be a subvariety of $Z\times_XY$ that dominates $Z$ and has $\dim T=\dim Z$. Such a subvariety exists because $\pi$ is dominant. The result follows
from the projection formula, using the functoriality of pullbacks and the assumption on $\pi^*\alpha$.
\end{proof}

Definition \ref{defn:upsef} does not seem practical for checking upsefness. It would be useful to give simpler criteria and a step in this direction is the following:

\begin{prop} \label{prop:bircritforbpf}
Let $\alpha\in N^k(X)$. Then $\alpha\in\upsef^k(X)$ if and
only if $\pi^*\alpha\in\Eff^k(Y)$ for any $\pi:Y\to X$ that is generically finite
onto its image (which can be a proper subset of $X$). 
\end{prop}
\begin{proof}By definition any upsef class satisfies the property in the proposition. Conversely, let $\alpha$ be a class which verifies said property. We use flattenings to check that it is universally pseudoeffective. Taking $\pi={\rm id}_X$, we see $\alpha\in\Eff^k(X)$.

Let $\pi:Y\to X$ be an arbitrary morphism of projective varieties.
Let $Z$ be the image of $\pi$ inside $X$, and denote by
$f:Y\to Z$ the induced dominant map and by $\imath:Z\to X$ the closed embedding. 
Let $\bar f:\bar Y\to\bar Z$ be a flattening of $f$ with generically finite morphism $\tau:\bar Z\to Z$ and $\bar Z$ nonsingular. By assumption, $(\imath\tau)^*\alpha\in\Eff^k(\bar Z)$.
Using the projection formula, it suffices to show that $\bar f^*(\imath\tau)^*\alpha\in\Eff^k(\bar Y)$. Hence without loss of generality we can assume that $\pi$ is flat and dominant and that $X$ is nonsingular. We want to show that $\pi^*\alpha\in\Eff^k(Y)$. Since $X$ is nonsingular and $\pi$ is flat, $\pi^*$ is defined on numerical groups and preserves
pseudoeffectivity. Furthermore $\pi^*(\alpha\cap[X])=(\pi^*\alpha)\cap[Y]$ by \cite[Theorem 3.2.(d)]{fulton84}. Consequently $\pi^*\alpha\in\Eff^k(Y)$. 
\end{proof}

\begin{rmk} \label{rmk:birforupsef} If resolutions of singularities exist, e.g. in characteristic zero, we can replace ``generically finite'' by ``birational'' in the Proposition.
\end{rmk}

\begin{exmple}
Suppose that $X$ is a smooth projective fourfold. Let $\alpha\in\Nef^1(X)$ and $\beta\in\Eff^1(X)$. Let $\delta\in N^2(X)$ be a class such that $\pi^*\delta\in\Eff^2(Y)$ for any $\pi:Y\to X$ generically finite and dominant (or just birational in characteristic 0). For example $\delta\in\Upsef^2(X)$. If $\gamma := \alpha \cdot \beta+\delta \in N^{2}(X)$ is nef, then $\gamma$ is universally pseudoeffective.
%(and in particular, if every component of the diminished base locus of $\beta$ has codimension $>$ 2) then $\gamma$ is universally pseudo-effective.

To see this, we apply Proposition \ref{prop:bircritforbpf}.  It suffices to consider morphisms $\pi: Y \to X$ that are generically finite onto their image.  By precomposing, we may furthermore assume that $Y$ is smooth.  Since codimension-two nef classes on smooth projective varieties of dimension at most three are pseudoeffective, we may assume $\dim Y = 4$. Then $\pi^{*}\gamma = \pi^{*}\alpha \cdot \pi^{*}\beta+\pi^*\delta$ is again pseudoeffective by assumption since $\pi^*\beta\in\Eff^1(Y)$ for dominant $\pi$.\qed
\end{exmple}

An interesting particular case of the above concerns an example of Fulton--Lazarsfeld \cite{fl82}, further investigated in \cite{pet09}:

\begin{exmple}\label{ex:flexupsef}Let $F$ be an an ample rank-two vector bundle on $\mathbb P^2$ sitting in an exact sequence 
$0\to \mathcal O(-n)^2\to\mathcal O(-1)^4\to F\to 0$
for sufficiently large $n$. The existence of such $F$ is explained in \cite{gie71}, or \cite[Example 6.3.67]{lazarsfeld04}. Let $$X=\mathbb P(\mathcal O\oplus F^{\vee}),$$ 
and let $$S=\mathbb P(\mathcal O)\subset X.$$ Fulton--Lazarsfeld (\cite[p.100]{fl82}) verify that $S$, which can also be seen as the zero section of the total space $X^0=X\setminus\mathbb P(F^{\vee})$ of $F$, has ample normal bundle (in fact $N_SX^0=N_SX=F$), but no multiple of $[S]$ moves in a nontrivial algebraic family inside $X^0$.
Peternell (\cite{pet09}) observes that the multiples of $S$ also do not move in $X$, and that $[S]$ is in the strict interior of $\Eff_2(X)$. Since $S$ has ample normal bundle, $[S]\in\Nef^2(X)$ (see \cite[Corollary 8.4.3]{lazarsfeld04}). 

We show that in fact $[S]$ belongs to the strict interior of $\upsef^2(X)$. Writing $[\mathbb P(\mathcal O)]\cdot[\mathbb P(F^{\vee})]=0$ in $X$, from the Groethendieck relation one can compute that
$$[S]=(\xi+\pi^*c_1(F))\cdot\xi+\pi^*c_2(F),$$
where $\pi:X\to \mathbb P^2$ is the bundle map, and where $\xi$ is the class in $N^1(X)$ of the relative $\mathcal O(1)$ Serre bundle. Observe that $\xi+\pi^*c_1(F)$ is ample. It is the relative $\mathcal O(1)$ for $(\mathcal O\oplus F^{\vee})\otimes \det(F)=\det(F)\oplus F$, which is ample. Also note that $\xi$ is effective, since $\mathcal O\oplus F^{\vee}$ has a section, and that $\pi^*c_2(F)$ is universally pseudoeffective, being the pullback of a positive multiple of the generator of $\Eff^2(\mathbb P^2)$. Then the previous example applies to the nef class $[S]$. Perturbing by a small multiple of the complete intersection $(\xi+(1-\epsilon)\pi^*c_1(F))^2$ for sufficiently small $\epsilon$, one sees that $[S]$ also belongs to the strict interior of $\upsef^2(X)$.  

The proof actually shows that if $S$ is a smooth projective surface, and $F$ is an ample vector bundle on $S$ of rank two, then the zero section of the total space of $F$ sitting as an open subset in $X=\mathbb P(\mathcal O\oplus F^{\vee})$ is in the strict interior of $\upsef^2(X)$.  
\qed
\end{exmple}

\begin{prop}
Let $\pi: X \to Y$ be an equidimensional morphism of projective varieties with relative dimension $d$, and with $Y$ smooth. Then $\pi_{*}\upsef^{k}(X) \subset \upsef^{k-d}(Y)$.
\end{prop}

\begin{proof}
Let $\alpha \in \upsef(X)$.  Suppose that $f: Z \to Y$ is a morphism from a projective variety $Z$ that is generically finite onto its image. Precompose to make $Z$ smooth if necessary. Consider the fiber product (where $Z'$ may be reducible)

\begin{equation*}
\begin{CD}
Z' @>>f'> X \\
@V\pi'VV        @VV\pi V\\
Z @>>f> Y
\end{CD}
\end{equation*}
Note that $\pi'$ is still equidimensional of relative dimension $d$, and by the dual of \cite[Proposition 6.2.(a)]{fulton84} (i.e. $\pi^*f_*\beta=f'_*\pi'^*\beta$ for all $\beta\in N_{k-d}(Z)$) we have $f^{*}\pi_{*} \alpha = \pi'_{*}f'^{*}\alpha$ in $N^{k-d}(Z)$.  Then as can be verified by pairing against any $P\in N^{\dim Z-(k-d)}(Z)$ we obtain
\begin{equation*}
(\pi'_{*}f'^{*}\alpha) \cap [Z] = \pi'_{*}(f'^{*}\alpha \cap [Z']),
\end{equation*}
which is pseudoeffective by the universal pseudoeffectivity of $\alpha$. 
\end{proof}

We end this subsection with a nontrivial computation of the universally pseudoeffective cones.

\begin{exmple}\label{ex:projbundleupsef}Let $X=\mathbb P_C(E)$, where $E$ is a vector bundle on a smooth curve $C$. Then $\upsef^k(X)=\Nef^k(X)$ for all $k$.
\end{exmple}
\begin{proof} Since the inclusion $\upsef^k(X)\subseteq\Nef^k(X)$ holds true in general, it is enough to show that every nef class is universally pseudoeffective. Consider the Harder--Narasimhan decomposition 
$E=E_0\supset E_1\supset\ldots\supset E_l=0$ with semistable successive quotients
$Q_i=\factor{E_{i-1}}{E_i}$ of slopes $\mu_i:=\frac{\deg Q_i}{{\rm rank}(Q_i)}$ forming an increasing sequence $\mu_1<\mu_2<\ldots<\mu_{l-1}$.

By \cite[\S7.1]{fl13z}, $$\Nef^k(X)=\langle\xi^k+\nu^{(k)}\xi^{k-1}f,\ \xi^{k-1}f\rangle,$$ where $\xi$ is the class of the relative Serre line bundle $\mathcal O_E(1)$ of the projective bundle map $\pi:X\to C$, where $f$ is the class of a fiber of $\pi$, and the $\nu^{(k)}$'s are computed in terms of the ranks 
and degrees of the $Q_i$. Moreover $\xi^{k-1}f=(\xi+af)^{k-1}f$ for any $a\in\mathbb R$. In particular it is an intersection of nef divisor classes, therefore universally pseudoeffective as well. It is then enough to show that $\xi^k+\nu^{(k)}\xi^{k-1}f$ is upsef.

Let $r={\rm rank}(Q_1)$. By \cite[\S7.1]{fl13z}, we have $\nu^{(k)}=-k\mu_1$ for $k\leq r$. Therefore $\xi^k+\nu^{(k)}\xi^{k-1}f=(\xi-\mu_1f)^k$ is the self-intersection of the nef class $\xi-\mu_1f$, which is upsef. In particular the statement of the example is
true when $E$ is semistable. Assume henceforth that $k>r$.

Let $h:Z\to C$ be any morphism from a projective variety $Z$, and let $F:Z\to X$ be a morphism such that $h=\pi\circ F$. Such $F$ corresponds to a
surjection $h^*E\to L$ onto a line bundle on $Z$, and then $L=F^*\mathcal O_E(1)$. By abuse we keep the notation $\xi=c_1(L)$ for its class in $N^1(Z)$ and the notation $f$ for the pullback of the fiber of $\pi$ to $Z$. We want to show that $\xi^k+\nu^{(k)}\xi^{k-1}f$ is psef. 

If $h^*E_1$ maps to $0$ inside $L$, then $F(Z)\subset\mathbb P(Q_1)\subset \mathbb P(E)$ and, since $k>r=\dim\mathbb P(Q_1)$, we have $\xi^k+\nu^{(k)}\xi^{k-1}f=0$. If not, then $h^*E_1$ maps onto $L\otimes\mathcal I$ for some nonzero ideal sheaf $\mathcal I$ on $Z$. One can show that $\mathcal I=\mathcal J\mathcal O_Z$, where $\mathcal J\subset\mathcal O_X$ is the ideal sheaf of $\mathbb P(Q_1)$ (see for example \cite[Proposition 2.4]{ful11}). The blow-up $\widetilde Z:=Bl_{\mathcal I}Z$ is the component of the fiber product $Z\times_X Bl_{\mathbb P(Q_1)}\mathbb P(E)$ that dominates $Z$. By \cite[Proposition 2.4]{ful11}, we have an induced morphism $\widetilde Z\to\mathbb P(E_1)$. Denote by $\xi_1$ the class of the Serre bundle for the map $\mathbb P(E_1)\to C$ and by $e$ the class of the exceptional divisor
on $Bl_{\mathbb P(Q_1)}\mathbb P(E)$. By abuse we keep the notation $\xi$, $\xi_1$, $e$, and $f$ for their pullbacks to $\widetilde Z$. 

We want to show that $\xi^k+\nu^{(k)}\xi^{k-1}f$ is psef on $Z$. By the projection formula it is enough to verify this after pulling back to $\widetilde Z$. By \cite[Proposition 2.4]{ful11}, we have 
\begin{equation}\label{eq:upsefproj}e(\xi-\mu_1f)^r=0\quad\mbox{and}\quad e=\xi-\xi_1.\end{equation}
Rewrite $$\xi^k+\nu^{(k)}\xi^{k-1}f=(\xi-\mu_1f)^r(\xi^{k-r}+(\nu^{(k)}+r\mu_1)\xi^{k-r-1}f).$$
By \eqref{eq:upsefproj}, given that $\xi-\mu_1f$ is a nef divisor class, it is enough to show that $$\xi_1^{k-r}+(\nu^{(k)}+r\mu_1)\xi_1^{k-r-1}f$$ is psef on $\widetilde Z$. This holds by induction because $\nu^{(k)}+r\mu_1=\nu_1^{(k-r)}$,
where $\nu_1^{(i)}$ give the nontrivial boundaries $\xi_1^i+\nu_1^{(i)}\xi_1^{i-1}f$ of $\Nef^i(\mathbb P(E_1))$ as follows from \cite[\S7.1]{fl13z}.
\end{proof}

\section{Basepoint free classes}\label{s:bpf}

One common way of constructing ``positive'' classes on $X$ is to take the class of a fiber of a morphism from $X$.  These classes are always nef and effective.  In fact, for any subvariety $V$ of $X$ we can find a fiber that has expected dimension of intersection with $V$.  In this section, we define the notion of a basepoint free class which satisfies similar properties.

\begin{defn}
Let $X$ be a projective variety of dimension $n$.  We say that $\alpha \in N_{n-k}(X)$ is a strongly basepoint free class if there is:
\begin{itemize}
\item an equidimensional quasi-projective scheme $U$ of finite type over $K$,
\item a flat morphism $s: U \to X$,
\item and a proper morphism $p:U \to W$ of relative dimension $n-k$ to a quasi-projective variety $W$ such that each component of $U$ surjects onto $W$
\end{itemize}
such that
\begin{equation*}
\alpha= (s|_{F_p})_{*}[F_{p}]
\end{equation*}
where $F_{p}$ is a general fiber of $p$.  Note that the resulting class is independent of the choice of fiber. We say that $p$ represents $\alpha$.

When $X$ is smooth, the basepoint free cone $\bpf^{k}(X)$ is defined to be the closure of the cone generated by such classes.
\end{defn}

\begin{rmk}\label{bpfgenpos}
The terminology indicates that the class $\alpha$ is ``basepoint free'' in the following sense: for every subvariety $V \subset X$ there is an effective cycle of class $\alpha$ that intersects $V$ in the expected dimension.  (To see this, let $d$ denote the codimension of $V$.  Then $s^{-1}V$ has codimension at least $d$ in $U$ by flatness, and $s^{-1}(V) \cap F_{p}$ has codimension at least $d$ in $F_{p}$.  Then $V \cap s(F_{p})$ has codimension at least $d$ in $s(F_{p})$ by upper-semicontinuity of fiber dimensions.)

Even though we define basepoint freeness using families of cycles, which gives it a ``covariant'' feel, we will show that $\bpf^{k}(X)$ is preserved by pullback between \emph{smooth} varieties, but $\bpf^{k}(X) \cap [X]$ is not preserved  by (arbitrary) pushforward.  Thus for \emph{smooth} varieties the basepoint free cone is really a ``contravariant'' cone.
\end{rmk}

It is clear that $\bpf^{k}(X) \subset \Nef^{k}(X)$ and $\bpf^{k}(X) \cap [X] \subset \Eff_{k}(X)$. Basepoint free classes also have an important property that we do not know for the pliant cone.  

\begin{lem}\label{lem:bpfflatpush}
Let $\pi: X \to Y$ be a flat morphism of smooth projective varieties.  Then $\pi_{*}\bpf^{k}(X) \subset \bpf^{k}(Y)$.
\end{lem}
\begin{proof}Immediate.\end{proof}

We next verify that, as suggested by Remark \ref{bpfgenpos}, for strongly basepoint free cycles we can exhibit explicit effective cycles that represent numerical intersections or pullbacks.  We will use these to verify that $\bpf^{k}$ satisfies the main properties desired for positive cones.  

\begin{lem} \label{lem:maintechnicalbpf}
Let $f:X\to Y$ be a projective morphism to a \emph{smooth} projective variety $Y$. Let $p:U\to W$ be a strongly bpf family on $Y$ with flat map $s:U\to Y$. For every top dimensional (effective) cycle $T$ on a general fiber $U_w$ of $p$ there exists a canonically defined (effective) cycle $X\cap_fT$ with support equal to $X\times_Y|T|$, and whose pushforwards represent:
\begin{enumerate}[i)]
\item  $f^*(s|_T)_*[T]\cap[X]\in N_*(X)$ on $X$.
\item $(s|_T)^*f_*[X]\cap[T]\in N_*(|T|)$ on $|T|$.
\item $(s|_T)_*[T]\cdot f_*[X]\in N_*(Y)$ on $Y$.
\end{enumerate}
In case $i)$, if $T=U_w$, then $X\cap_fU_w=U'_w$, where $U'=U\times_YX$. In particular, if $X$ is also smooth, then $f^*\bpf^k(Y)\subset\bpf^k(X)$.
\end{lem} 

\begin{comment}
\begin{shaded} To Mihai, Feb 17: What if we lose smoothness of $Y$?  I would hope we can do the following.  First, we should lose the ability to work with components of bpf families.  But for the entire cycle, we should obtain a cycle-theoretic intersection on $X$ that pushes forward to the right rational equivalence class on $Y$ (that is, we would keep iii on the level of Chow).  Furthermore if we ever construct flat pullbacks this should refine the flat pullback (so maybe we could keep i).
\end{shaded}

\begin{shaded}To Brian, Feb17: So assume $X,Y$ singular projective, $f$ proper, $p$ strongly bpf.

First the statement of $iii)$ makes sense only when $Y$ is smooth, or one of $f(X)$ or $s(T)$ is regularly embedded in $Y$. But then it is probably true.

Statement $ii)$ should always be true numerically on $U_w$ instead of $T$, and in a Chow version on $U$ with $s$ replacing $s|_T$.

Statement $i)$ should be true when $f$ is l.c.i, in Chow groups when $f$ is flat, and numerically if we prove that flat pullbacks are numerical.
\end{shaded}
\end{comment}

\begin{proof}Let $\Gamma_f:X\to X\times Y$ be the graph of $f$. Since $Y$ is smooth, $\Gamma_f$ is a regular embedding. Consider the flat base change map $X\times U\to X\times Y$. For general $w\in W$, the arguments of Remark \ref{bpfgenpos} and the regularity of the embedding $\Gamma_f$ show that $X\times_YU_w=X\times_{X\times Y}(X\times U_w)$ is equidimensional of the expected dimension or empty. The same is true for any top dimensional cycle $T$ on $U_w$.

We are in a setting of proper intersection (cf. \cite[\S7.1]{fulton84}). Then by counting every component of $X\times_Y|T|$ with its (positive) multiplicity of intersection (again in the sense of \cite[\S7.1]{fulton84}) we get a canonically defined effective cycle $X\cap_fT$ supported on it and representing $X\cdot_{\Gamma_f}(X\times T)$ in the sense of \cite[\S6.2]{fulton84}. But this is $[X]\cdot_f[T]=f^![T]$ as in \cite[Definition 8.1.2]{fulton84}. Its pushforward to $X$ is $f^!(s|_T)_*[T]=[X]\cdot_f(s|_T)_*[T]$ by the projection formula \cite[Proposition 8.1.1.(c)]{fulton84}.

Since $Y$ is nonsingular, by \cite[Example 15.2.16.(b)]{fulton84}, there exists a Chern polynomial with $\mathbb Q$-coefficients such that $P\cap[Y]=(s|_T)_*[T]$. Then by \cite[Example 8.1.6 and Corollary 8.1.3]{fulton84}, $$[X]\cdot_f(P\cap[Y])=(f^*P\cap[X])\cdot_f[Y]=f^*P\cap[X]\in A_*(X).$$
But the numerical class of $f^*P\cap[X]$ is by definition $f^*(s|_T)_*[T]\cap[X]\in N_*(X)$. The pushforward to $|T|$ is analogous, and the pushforward to $Y$ is computed by the projection formula.

When $X$ is also smooth and $T=U_w$, then it is enough to observe that $X\cdot_{\Gamma_f}(X\times U_w)=[U'_w]$ which is true because $\Gamma_f$ is a regular embedding. (See also the proof of \cite[Corollary 8.1.3]{fulton84}).
\end{proof}

\begin{cor}
Let $\pi: X \to Y$ be a morphism of projective varieties with $Y$ smooth.  Let $p: U \to W$ be a strongly bpf family on $Y$ of class $\alpha$, with flat map $s: U \to Y$.  Suppose that $V$ is a cycle on $X$ whose support is contracted by $\pi$.  Then the class $[V] \cdot \pi^*\alpha$ is represented by a cycle on $X$ whose support is contracted by $\pi$.
\end{cor}

\begin{proof}
For $w$ general, we may suppose that for any component $V_{i}$ of $V$ the set-theoretic intersection of $\pi(V_{i})$ with $\Supp(U_{w})$ has the expected dimension.  Consider the intersection cycle $V_{i} \cap_{f} U_{w}$ as defined in Lemma \ref{lem:maintechnicalbpf}.  Since each component of this cycle has the expected codimension, the map from any component of this set to $f(V_{i}) \cap \Supp(U_{w})$ has positive dimensional fibers.
\end{proof}

\begin{cor}\label{cor:bpfstableint}
If $X$ is a smooth projective variety, then the intersection of basepoint free classes on $X$ is basepoint free.
\end{cor}

\begin{proof}
Suppose that $p: U \to W$ and $p': U' \to W'$ are strongly bpf families on $X$.
Consider the diagram
\begin{equation*}
\begin{CD}
U \times_{X} U' @>>> U \\
@VVV        @VVsV\\
U' @>>s'> X
\end{CD}
\end{equation*}
The composed map $U \times_{X} U' \to X$ is flat and the family $p \times p': U' \times_{X} U \to W \times W'$ represents the intersection class by Lemma \ref{lem:maintechnicalbpf}.
\end{proof}

\begin{lem}\label{lem:plbpfupsef}
Let $X$ be a smooth projective variety.  Then $\pl^{k}(X) \subseteq \bpf^{k}(X) \subseteq \upsef^{k}(X)$.
\end{lem}

\begin{proof}
To see the first inclusion, by Lemma \ref{lem:maintechnicalbpf} it suffices to show that $\Eff^{k}(\mathbb{G}) = \bpf^{k}(\mathbb{G})$ for a Grassmannian $\mathbb{G}$.  But we can construct flat families representing elements $\Eff^{k}(\mathbb{G})$ using the group action.  More precisely, suppose $Z$ is a Schubert variety on $\mathbb{G}(V)$.  Set $W=\mathbb{P}GL(V)$, and consider the family $U \subset W \times \mathbb{G}(V)$ whose fiber over $g \in W$ is $gZ$.  Then the projection $s: U \to \mathbb{G}(V)$ is flat since it is $\mathbb{P}GL(V)$-equivariant, showing that $[Z] \in \bpf(\mathbb{G}(V))$.

By Lemma \ref{prop:domupsef}, we may check containment in $\upsef$ after pulling back via a dominant map.  In particular, by passing to an alteration to verify the $\upsef$ property it suffices to consider pullbacks to smooth varieties.  The second inclusion then follows from Lemma \ref{lem:maintechnicalbpf}.
\end{proof}

\begin{cor}\label{cor:bpffullsal}
Let $X$ be a smooth projective variety.  Then $\bpf^{k}(X)$ is a full-dimensional salient cone.
\end{cor}

\begin{exmple}\label{ex:bpfhom}
The proof of Lemma \ref{lem:plbpfupsef} shows that $\bpf^k(H)=\upsef^k(H)=\Eff^k(H)\subset\Nef^k(X)$ for any smooth projective homogeneous space $H$. When $H$ is one of the examples of abelian varieties of \cite{delv11}, the last inclusion may be strict.
\end{exmple} 

\begin{exmple}\label{ex:projbundlebpf}Let $E$ be a vector bundle over $\mathbb P^1$, and let $X=\mathbb P(E)$. Then $\bpf^k(X)=\Nef^k(X)$ for all $k$.

We follow the notation of Example \ref{ex:projbundleupsef} and do induction on the number of semistable factors of $E$. When $E$ is semistable, then $X$ is isomorphic to a product.  The generators $(\xi-\mu_1f)^k$ and $\xi^{k-1}f$ of $\Nef^k(X)$ are both pliant and we conclude by Lemma \ref{lem:plbpfupsef}. 

For general $E$, the same argument as in the semistable case works as long as $k\leq r={\rm rank}\, Q_1$. It is enough to check that $\xi^k+\nu^{(k)}\xi^{k-1}f$ is strongly basepoint free for all $k>r$.
Let $Z:={\rm Bl}_{\mathbb P(Q_1)}\mathbb P(E)$ with blow-down map $\sigma:Z\to X$ and bundle map $\eta:Z\to\mathbb P(E_1)$. On $Z$ we have
$$\sigma^*(\xi^k+\nu^{(k)}\xi^{k-1}f)=\sigma^*(\xi-\mu_1f)^r\cdot\eta^*(\xi_1^{k-r}+\nu_1^{(k-r)}\xi_1^{k-r-1}f_1).$$
Since we work over $\mathbb P^1$, the varieties $X$ and $Z$ are toric. Since the class $\xi-\mu_1f$ is nef, it is also semiample and in particular strongly basepoint free. The class $\xi_1^{k-r}+\nu_1^{(k-r)}\xi_1^{k-r-1}f_1$ is strongly basepoint free by induction. From (the proofs of) Lemma \ref{lem:maintechnicalbpf} and Corollary \ref{cor:bpfstableint}, it follows that $\sigma^*(\xi^k+\nu^{(k)}\xi^{k-1}f)$ is strongly basepoint free. 

Again because we work over $\mathbb P^1$, the bundle $E$ is split. Then we also have an inclusion $\mathbb P(E_1)\subset\mathbb P(E)=X$ such that $\mathbb P(E_1)\cap\mathbb P(Q_1)=\emptyset$ in $X$ and $[\mathbb P(E_1)]=(\xi-\mu_1f)^r$.
Thus $Z={\rm Bl}_{\mathbb P(Q_1)}\mathbb P(E)$ contains the copy $\sigma^{-1}\mathbb P(E_1)$ of $\mathbb P(E_1)$ that does not meet the exceptional locus of $\sigma$, and with numerical class $\sigma^*(\xi-\mu_1f)^r$. %By applying Lemma \ref{lem:maintechnicalbpf} we find a class representing the strongly basepoint free class $$\sigma^*(\xi^k+\nu^{(k)}\xi^{k-1}f)=[\sigma^{-1}\mathbb P(E_1)]\cdot\eta^*(\xi_1^{k-r}+\nu_1^{(k-r)}\xi_1^{k-r-1}f_1)$$ 
%that does not meet the exceptional locus of $\sigma$.

Furthermore, $\sigma^{-1}\mathbb P(E_1)$ is a complete intersection of $r$ sections of $\xi-\mu_1f$ corresponding to a basis of the trivial component of $E\otimes\mathcal O_{\mathbb P^1}(-\mu_1)$.  Let $p$ denote the corresponding basepoint free family.  It follows that the general element of the basepoint free family constructed by intersecting the pullback of $p$ and the pullback basepoint free family from $\mathbb{P}(E_{1})$ does not meet the exceptional locus of $\sigma$. Up to shrinking the base, we see this as a family of cycles on $X$ representing the class $\xi^k+\nu^{(k)}\xi^{k-1}f$, which is then also strongly basepoint free.
\qed
\end{exmple}

The following example shows that the basepoint free cone of curves coincides with the nef cone for any smooth Mori Dream Space $X$.  The curves we construct come from small $\mathbb{Q}$-factorializations of $X$ which extract the Zariski decomposition of divisors on $X$.

\begin{exmple} \label{ex:bpfandmds}
Let $X$ be a smooth Mori Dream Space of dimension $n$ (for example, a toric variety).  We prove that $\bpf^{n-1}(X) = \Nef^{n-1}(X)$.

Recall that by \cite{bdpp04} the cone $\Nef^{n-1}(X)$ is generated by the positive products $\langle D^{n-1} \rangle$ as $D$ varies over all movable divisors (where $\langle - \rangle$ denotes the positive product).  We can turn this into a geometric construction as follows.  Fix an ample divisor $A$ on $X$.  Let $\alpha$ be a class on an extremal ray of $\Nef^{n-1}(X)$ and let $D$ be a divisor on the boundary of the movable cone of divisors such that the rays spanned by $\langle (D + \epsilon A)^{n-1} \rangle$ approach the ray spanned by $\alpha$.  Then the same is true if we replace $A$ by any big divisor $B$ using the continuity of the positive product.

There is a small birational contraction $\phi_{D}: X \dashrightarrow X'$ so that $D' := \phi_{D*}D$ is a semiample divisor; for simplicity, we rescale $D$ so that we may suppose $D'$ is basepoint free.  Let $W$ be a common smooth resolution of $X$ and $X'$ with birational maps $\psi: W \to X$ and $\psi': W \to X'$.  Fix an ample divisor $A'$ on $X'$ and let $B$ be the strict transform class on $X$.  Note that for sufficiently small $\delta > 0$, $X'$ is the minimal model for $D + \delta B$ and $D' + \delta A'$ is the pushforward of this class.  Then for sufficiently small $\delta$, we have
\begin{align*}
\langle (D + \delta B)^{n-1} \rangle & = \psi_{*} \langle \psi^{*}(D+\delta B)^{n-1} \rangle \\
& = \psi_{*} \langle \psi'^{*}(D' + \delta A')^{n-1}\rangle \\
& = \phi_{D*}^{-1} \langle (D' + \delta A')^{n-1} \rangle.
\end{align*}

Define a flat family of curves $p_{\delta}: \mathcal{C} \to W$ on $X'$ by taking complete intersections of $n-1$ general elements of a very ample linear series which is a multiple of $D' + \delta A'$ (for sufficiently small rational $\delta$).  Let $U \subset X'$ be the open subset on which $\phi_{D}$ is an isomorphism.  Note that the complement of $U$ has codimension-two.  Since $\mathcal{C}$ defines a flat family of curves, the preimage of $U$ has complement of codimension-two in $\mathcal{C}$.  Since $p_{\delta}$ has fiber dimension one, this set does not dominate the base $W$.  Thus by removing a proper closed subset from $W$ we obtain a family of curves $p_{\delta}^{0}: \mathcal{C}^{0} \to W^{0}$ whose map to $X'$ is flat and factors through $U$.

The strict transform of a general member of this family to $X$ defines a basepoint free curve class.  Since this strict transform avoids the exceptional locus of the map $\phi_{D}$, we see that the limit of the rays spanned by the fibers of $p_{\delta}^{0}$ as $\delta$ goes to $0$ is the same as $\alpha$, finishing the proof. 
\qed

\end{exmple}

Finally, we recall the example of \cite{bh15} which proves that universally pseudo-effective classes need not be basepoint free.

\begin{exmple} \label{upsefnotbpf}
Let $X$ be a smooth projective variety and suppose that some multiple of a class $\alpha \in N_{2}(X)$ is represented by an irreducible surface $S$.  Then the Hodge index theorem shows that the pairing
\begin{equation*}
N^{1}(X) \times N^{1}(X) \to \mathbb{R}; \qquad \qquad (D_{1}, D_{2}) \mapsto D_{1} \cdot D_{2} \cdot \alpha
\end{equation*}
must have at most one distinct positive eigenvalue.

Using a combinatorial argument \cite[Section 5]{bh15} constructs a smooth toric fourfold $X$ and a nef surface class $\alpha$ satisfying the following properties:
\begin{enumerate}
\item $\alpha$ generates an extremal ray of $\Nef^{2}(X)$.
\item The signature of the pairing on $N^{1}(X) \times N^{1}(X)$ given by intersection against $\alpha$ has three distinct positive eigenvalues.
\end{enumerate}
We show that $\alpha$ is not a limit of sums of classes which are nef and represent irreducible $\mathbb{R}$-cycles.  In particular $\alpha \not \in \bpf^{2}(X)$, but by Example \ref{ex:sphericalupsef} we have $\alpha \in \upsef^{2}(X)$.
We will need the following useful cone lemma:

\begin{lem}\label{irraprox}Let $C$ be a closed full-dimensional salient convex cone which is the closure of a cone generated by a set $\{c_i\}$ inside a finite dimensional vector space. Let $\alpha\in C$ span an extremal ray. Then there exists a subsequence $\{c_j\}$ of $\{c_i\}$ and positive real numbers $r_j$ such that $\alpha=\lim_{j\to\infty}r_jc_j$.\end{lem}

Suppose that $\alpha$ is in the closure of the cone generated by nef irreducible cycles.  By (1), $\alpha$ spans an extremal ray of this cone.  Applying Lemma \ref{irraprox} to this cone (considered as a full-dimensional cone inside of the vector space it spans), we see that $\alpha$ is a limit of rescalings of irreducible nef cycles, a contradiction to (2).
\end{exmple}

\section{Questions}\label{s:questions}

\subsection{The pliant cone} \label{ss:pliantques}

We have defined the pliant cone in terms of monomials in Schur classes of globally generated bundles on $X$. The motivation for using Schur classes is that with this definition it is easy to see that $\Eff^k(\mathbb G)=\pl^k(\mathbb G)$ for any Grassmann variety $\mathbb G$. This is used in the proof of Lemma \ref{lem:ci intpliant}, where we say that since complete intersections are big, they are also in the interior of the pliant cone on $\mathbb G$.

\begin{ques}What happens if we change the definition of the pliant cone to include only monomials in Chern classes, or only monomials in dual Segre classes of vector bundles?
\end{ques}

\begin{exmple}Let $X=G(2,4)$ and let $Q$ be the universal quotient bundle of rank $2$ on $X$. Then $\Eff^2(X)$ is generated by the Schur classes $s_{(1,1)}(Q)=c_1^2(Q)-c_2(Q)$ and $s_{(2)}(Q)=c_2(Q)$ of $Q$. Note that $s_{(1,1)}(Q)=s_2(Q^{\vee})$, i.e. the second dual Segre class.

If we use only monomials in Chern classes of $Q$, then we get the smaller cone generated by $c_1^2(Q)$ and $c_2(Q)$. While if we use only monomials in dual Segre classes of $Q$, we obtain the ``complementary'' cone generated by $s_2(Q^{\vee})=c_1^2(Q)-c_2(Q)$ and $s_1^2(Q^{\vee})=c_1^2(Q)$.

However, if we also use $R$, the dual of the universal subbundle of rank $2$, so that we have an exact sequence
$0\to R^{\vee}\to\mathcal O_X^{\oplus 4}\to Q\to 0$ and $R$ is globally generated, then $c_1(R)=c_1(Q)$ and $c_2(R)=s_2(Q^{\vee})$. Therefore $\Eff^2(X)$ is generated by the Chern monomials $c_2(R)$ and $c_2(Q)$, or by the dual Segre monomials $s_2(Q^{\vee})$ and $s_2(R^{\vee})$. \qed
\end{exmple}

It is interesting to see if in higher codimension one can express the classes of Schubert cycles on Grassmannians as Chern monomials and as Segre monomials of globally generated 
bundles obtained by tensoring Schur functors $S_{\lambda}(R)\otimes S_{\mu}(Q)$.

Example \ref{ex:grassmann} describes the pliant cone for products of Grassmann varieties.  The next example to consider is homogeneous varieties.

\begin{ques}Let $f:X\to H$ be a morphism to a projective nonsingular homogeneous variety (e.g. partial flag variety, or an abelian variety) and let $\alpha\in\Eff^k(H)$.
Is it true that $f^*\alpha\in\pl^k(X)$? Equivalently, is it true that $\Eff^k(H)=\pl^k(H)$?
\end{ques}

A property of nefness is that it can be checked on dominant covers. It is not clear that the same is true for pliancy.

\begin{ques}Let $X$ be a (smooth) (complex) projective variety, and let $\pi:Y\to X$ be a dominant projective morphism. Assume that $\pi^*\alpha\in\pl^k(Y)$. Then does $\alpha\in\pl^k(X)$?
\end{ques}

By analogy with the other notions of positivity, we ask:

\begin{ques}Let $\pi: Y \to X$ be a flat morphism from a projective variety $Y$ to a smooth projective variety $X$ of relative dimension $d$.  Suppose that $\alpha \in \pl^{k+d}(Y)$.  Then is $\pi_{*}\alpha \in \pl^{k}(X)$?
\end{ques}

%In the definition of pliant classes we use global generation for its relation to geometry. It is natural to ask what happens for other positivity notions for vector bundles.

\subsection{Chern classes for ample vector bundles} \label{ss:Chernampleques}

Another way to modify the definition of the pliant cone is to allow arbitrary nef vector bundles, instead of just globally generated ones.  However, it is not clear if the resulting cone consists of effective classes.  The following question is also posed in \cite{fl83} and \cite[\S6]{delv11}. 

\begin{ques}\label{q:Chernupsef}Let $E$ be a nef (or ample) vector bundle on a projective variety, and let 
$\lambda$ be a partition. Is the Schur class $s_{\lambda}(E)\cap[X]$ pseudoeffective?  Is this true for Chern classes?
\end{ques}

Since nefness is preserved by pullback, this is the same as asking if $s_{\lambda}(E)$ is universally pseudoeffective. The answer is yes for dual Segre classes $s_k(E^{\vee}):=s_{(1^r)}(E)$, i.e. when $\lambda$ is the partition $(1,\ldots,1)$ of $k$. The answer is also known to be yes for the Chern classes $c_k(E)$ when $k\in\{1,\dim X-1,\dim X\}$. Quite generally it is a consequence of a result of Bloch--Gieseker (\cite[Theorem 8.2.1]{lazarsfeld04}) that $s_{\lambda}(E)\in\Nef^k(X)$ for any partition $\lambda$ of length $k$. The first unknown case is $c_2$ for nef bundles on fourfolds.
The issue here is that if say $E$ is ample, then ${\rm Sym}^m(E)$ (or in characteristic zero also $E^{\otimes m}$) is globally generated for large $m$, but $c_2(E)$ is not a scalar multiple of the pliant classes $c_2({\rm Sym}^m(E))$ or $c_2(E^{\otimes m})$.
It is also true that if $E$ is $p$-ample in characteristic $p>0$ (cf. \cite{gie71}), then $c_k(E)$ is pseudoeffective for all $k$, since $c_k(E)$ is proportional to $c_k(E^{p^e})$, and the iterated Frobenius pullback $E^{p^e}$ is globally generated for large $e$. Gieseker \cite{gie71} constructs an example of an ample bundle on $\mathbb P^2$ that is not $p$-ample.

\begin{ques}Let $E$ be a nef (i.e. $\mathcal O_{\mathbb P(E)}(1)$ is a nef line bundle) vector bundle on $X$. Is it true that $s_{\lambda}(E)$ is pliant?
\end{ques}

Another question related to Question \ref{q:Chernupsef} is:

\begin{ques}Let $X$ be a smooth projective variety and let $Y$ be a closed subvariety with nef (or ample) normal bundle. Is $[Y]$ universally pseudoeffective?
\end{ques}

If $f:Z\to X$ is a morphism of projective varieties, then $f^*[Y]=p_{1*}(\Gamma_f\cdot[Z\times Y])$, where $p_1:Z\times X\to Z$ is the first projection, and $\Gamma_f\subset Z\times X$ is the graph of $f$. Given that $N_{Z\times Y}Z\times X=(p_{1}|_Y)^*N_YX$ is still nef, the universal pseudoeffectivity of Chern classes of nef bundles would imply the pseudoeffectivity of $\Gamma_f\cdot[Z\times Y]$ in view of \cite[Proposition 6.1.(b)]{fulton84}, hence also that of $f^*[Y]$. It is known (\cite[Corollary 8.4.3]{lazarsfeld04}) that $[Y]$ is nef.

If Schur polynomials in Chern classes of nef bundles are not universally pseudoeffective, then it is interesting to ask what cone they generate.  In particular, it would be very interesting if they generate the entire nef cone.  The following question is a step in this direction:

\begin{ques}Let $E$ and $F$ be nef vector bundles on $X$, and let $\lambda$ and $\mu$ be partitions. Is $s_{\lambda}(E)\cdot s_{\mu}(F)$ nef?
\end{ques}

A positive answer, applied to the examples of \cite{delv11}, would show that $\Nef^k(X)$ is not the closure of the cone generated by classes $s_{\lambda}(E)$ with $E$ nef and $\lambda$ a partition of $k$. A negative answer would show that the answer to Question \ref{q:Chernupsef} is also no.

\subsection{The universally pseudoeffective cone}

By considering all maps of projective varieties $\pi:Y\to X$ in the definition of universal pseudoeffectivity, we guaranteed that this notion is preserved by pullback and in particular stable under products, thus removing some of the pathologies of nefness exhibited in \cite{delv11}. 

\begin{ques}Let $\alpha\in N^k(X)$ be such that $\imath^*\alpha\in\Eff^k(Y)$ for all embeddings of closed subvarieties $\imath:Y\hookrightarrow X$. Then is $\alpha\in\Upsef^k(X)$?
\end{ques}

A slightly weaker version of this also appears in \cite{delv11}:

\begin{ques}Let $X$ be a smooth (complex) projective variety. Let $\alpha\in N^k(X)$ be such that $\alpha\cdot[Y]$ is pseudoeffective for any closed subvariety $Y\subset X$.
Is it true that $\alpha\in\Upsef^k(X)$?
\end{ques}

This is weaker than the previous question because the pseudoeffectivity of $\alpha\cdot[Y]=\imath_*\imath^*\alpha$ is only implied by that of $\imath^*\alpha$. We expect that the answer to the next question is no, but a counterexample is missing:

\begin{ques}Let $X$ be a (smooth) (complex) projective variety. Is $\Upsef^k(X)=\Eff^k(X)\cap\Nef^k(X)$?
\end{ques}

June Huh asks whether a stronger statement is true:

\begin{ques}Let $X$ be a smooth complex projective variety. Is $\Eff^k(X)\cap\Nef^k(X)$  the closure of the cone generated by classes $\alpha$ such
that for each subscheme $T\subset X$ there exists a $\mathbb Q$-cycle $Z$ whose support meets $T$ properly and with $[Z]=\alpha$?
\end{ques}

The smallest dimension where a counterexample might exist is $n=4$ and $k=2$. It is also expected that a counterexample should exist in any birational equivalence class of sufficiently large dimension.

\subsection{Curves}

Curves provide an important test case for understanding the various positive cones.  Let $X$ be a smooth projective variety of dimension $n$ and define the cone $\mathrm{CI}^{n-1}(X) \subset N^{n-1}(X)$ to be the cone generated by complete intersections of nef divisors.  Note that $\mathrm{CI}^{n-1}(X)$ is a positive cone of dual classes: it is full-dimensional, salient, nef, and contains complete intersections of ample divisors in its interior.

We have $\Nef^{n-1}(X) = \Upsef^{n-1}(X)$ and
\begin{equation*}
\mathrm{CI}^{n-1}(X) \subset \pl^{n-1}(X) \subset \bpf^{n-1}(X) \subset \Nef^{n-1}(X).
\end{equation*}
The question is whether any other equalities hold.  Example \ref{ex:plandflop} shows that there can be a strict containment $\mathrm{CI}^{n-1}(X) \subsetneq \pl^{n-1}(X)$, and Example \ref{ex:bpfandmds} gives many examples where $\mathrm{CI}^{n-1}(X) \subsetneq \bpf^{n-1}(X)$.  However, one wonders if for example the $(n-1)$-dual Segre class of an ample vector bundle (corresponding to the partition $\lambda = (1^{n-1})$) is contained in $\mathrm{CI}^{n-1}(X)$.

Recall that by the main result of \cite{bdpp04}, the nef cone of curves is generated by pushforwards of complete intersections of ample divisors on birational models.  It is sometimes inconvenient that this cone does not coincide with $\mathrm{CI}^{n-1}(X)$, and it would be very interesting if the complete intersection cone could be recovered naturally from a different perspective.

\subsection{Other positive cones}
There are many other ways to construct positive cones.  We have already discussed several variations of the definition of the pliant cone: one can use a smaller set of classes (such as dual Segre classes; see Section \ref{ss:pliantques}) or a larger set of bundles (such as all ample bundles; see Section \ref{ss:Chernampleques}).  It would be very interesting to have a better understanding of the resulting cones.

One can also define many minor variations of the basepoint free cone.  For example, one can define $\bpf_{k}(X)$ by taking the cone generated by classes of $k$-dimensional components of arbitrary flat families of subschemes.  The resulting cone also contains the pliant cone, but it is not clear how it differs otherwise.

Finally, there are many other notions of positivity in the literature which may be suitable for constructing cones.  First, \cite{har70} defines an ample subvariety of a smooth variety $X$ to be an l.c.i.~subscheme with ample normal bundle.  Unfortunately, it is not clear that the classes of such subvarieties span $N^{k}(X)$; indeed, this is a very subtle question even just for l.c.i.~subvarieties.  An alternative is proposed by \cite{ott12}, which defines positivity by the $q$-ampleness of the exceptional divisor on a blow-up.  Ottem has communicated to us a sketch of the fact that the classes of such subvarieties span a full-dimensional cone in $N_{k}(X)$.

Alternatively, one can focus on the positivity of currents as discussed in \cite{bdpp04} and \cite{delv11}.  Unfortunately, to relate the resulting cones with cycles in higher codimensions it seems that one often must assume some version of the Hodge Conjecture.  Nevertheless, it would be useful to see some different approaches to positivity from this perspective.

\nocite{*}
\bibliographystyle{amsalpha}
\bibliography{poscones}

\end{document}